\documentclass[12pt,reqno]{article}
\usepackage{amsmath, amssymb}
\usepackage{latexsym}
\usepackage{amsfonts}
\usepackage{caption}
\usepackage{amsmath}
\usepackage{amssymb}
\usepackage{amsthm}
\usepackage{subcaption}
\usepackage{float}
\usepackage{color}
\usepackage{enumitem}
\usepackage[utf8]{inputenc}
\usepackage{tikz}
\usetikzlibrary{shapes.geometric, arrows.meta, positioning}

\everymath{\displaystyle}
\usepackage{graphicx}
\usepackage{amscd}
\usepackage{amsbsy}
\usepackage[colorlinks, hyperindex, bookmarks, linkcolor=blue, citecolor=blue, urlcolor=blue]{hyperref}

\newcommand{\ds}{\displaystyle}

\def\beq{\begin{equation}}
	\def\eeq{\end{equation}}

\def\ba{\begin{array}{ll}}
	\def\ea{\end{array}}
\def\ds{\displaystyle}

\def\CC{\rm \hbox{C\kern-.56em\raise.4ex \hbox{$\scriptscriptstyle
			|$}\kern+0.5 em}}

\newtheorem{theorem}{Theorem}

\newtheorem{definition}{Definition}
\newtheorem{proposition}{Proposition}

\newtheorem{remark}{\bf Remark}

\begin{document}

	\title{A degenerate reaction-diffusion SIR model in interconnected regions}
	
	\author{ Omar Elamraoui\thanks{Laboratoire MISI, FST Settat, Univ. Hassan I, 26000 Settat,  Morocco (oelamraoui34@gmail.com)}  \and Jawad Salhi \thanks{MAMCS Group, FST Errachidia, Moulay Ismail University of Meknes,  Morocco (j.salhi@umi.ac.ma)}\and Abderrahim  Zafrar \thanks{Department of Mathematics, Faculty of Sciences Dhar El Mahraz, Sidi Mohamed Ben Abdellah University, Fez,  Morocco (zafrar.abd@gmail.com)} }
	\medskip
	\date{ }
	\maketitle
	\begin{abstract}
		This paper presents a novel time-space SIR (Susceptible–Infected–Recovered) model for simulating infectious disease dynamics in two interconnected regions. The model is formulated as a coupled reaction-diffusion system with boundary conditions that dynamically switch from Robin to Neumann types, effectively modelling policy-driven interventions such as lockdowns. A key innovation lies in the incorporation of degenerate diffusion, arising from vanishing population density, which significantly influences transmission behaviour near regional borders. The well-posedness of the model is rigorously established using the Faedo–Galerkin method, ensuring the existence, uniqueness, and positivity of weak solutions. Numerical simulations, performed using the Finite Volume Method, validate the theoretical findings and demonstrate the impact of migration and mobility restrictions on epidemic progression. This framework offers valuable insights for understanding and controlling disease spread in spatially heterogeneous and interconnected settings.
	\end{abstract}
	
	
	\section{Introduction}
Mathematical modelling has become an indispensable tool for understanding, predicting, and controlling the spread of infectious diseases. 
Among the earliest and most influential 
frameworks are compartmental models, particularly the classical SIR model proposed 
by Kermack and McKendrick~\cite{kermack1927contribution}. This model divides the population into 
three compartments, usually denoted by $S$ (density of susceptible individuals),
     $I$ (density of infected individuals) and 
     $R$ (density of recovered (or removed) individuals).
In its simplest form, the SIR model is described by a system of ordinary differential equations (ODEs). 
Thanks to its analytical tractability and interpretability, it has been widely used to study various 
epidemics, including measles, influenza, and COVID-19. However, real-world epidemics involve 
complex mechanisms that classical ODE-based models cannot fully capture. These mechanisms include spatial heterogeneity, where disease prevalence varies across locations; population mobility and migration, which connect geographically separated regions; and policy-driven interventions, such as vaccination, quarantine, and lockdown measures.
\medskip
\\
To address these complexities, researchers have developed a wide range of mathematical models, many of which extend the traditional SIR framework. Examples include the Lotka-Volterra model
for predator-prey dynamics~\cite{brauer2012mathematical,murray2002mathematicalbiology},  fractional time-delay epidemic models~\cite{kumar2021fractional,sweilam2020optimal},  
and various stochastic epidemic models~\cite{greenwood2009stochastic,metz1978epidemic,castillo1993stochastic,ndii2017stochastic}.  
These approaches, often formulated as partial differential equations (PDEs), make it possible to 
analyze the spatio-temporal dynamics of host populations and disease propagation, offering 
a richer description than ODE-based models~\cite{essoufi2022boundary,elamraoui2023spatio,allen2008asymptotic,bertuzzo2010spatially,cantrell2004spatial}.
\medskip
\\
While a substantial portion of the existing literature addresses SIR epidemic models for non-degenerate partial differential equations, the mathematical study of these models with degenerate equations remains largely unaddressed. This class of equations presents a unique set of challenging mathematical problems due to the loss of uniform ellipticity in the operators, a phenomenon that can occur either on parts of the boundary or within a subregion of the spatial domain. The operators typically involve variable diffusion coefficients that are not uniformly elliptic over the entire domain, even if they remain so on compact subsets at a positive distance from the degeneracy points. Consequently, the analysis of these systems introduces significant analytical difficulties, particularly regarding the well-posedness of the associated evolution equations. In such contexts, classical tools are often insufficient, necessitating the development of an alternative functional framework, such as weighted Sobolev spaces, which differs from the usual setting for classical PDEs.
\medskip
\\
Human migration represents another critical factor in epidemic dynamics. It shapes cultural, demographic, and economic landscapes across the globe and significantly influences disease transmission. With globalization, international migration exceeded 281 million people worldwide in 2020~\cite{menozzi2021international}. As a result, many epidemic models have incorporated the effects of migration and immigration~\cite{brauer2001models,wang2015epidemic}. 
\\
In particular, various SIR-based models accounting for infected immigration were studied by the authors in 
\cite{brauer2001models}. The authors in \cite{mccluskey2004global} extended 
this by analyzing models with exposed individuals entering the population, while 
in \cite{sigdel2014global}, the authors developed an SEI model that considers 
immigration across all compartments. Furthermore, the authors in ~\cite{sigdel2014disease} explored the 
relationship between disease dynamics and the birthplace of migrant workers.
\medskip
\\
Despite these contributions, most existing epidemic models assume constant mobility rates
and neglect the impact of policy-driven mobility restrictions, such as lockdowns, 
which have become central to epidemic control strategies. These limitations motivate the 
development of more realistic frameworks incorporating spatial heterogeneity, 
migration and degenerate diffusion.
In this work, we propose a novel time-space SIR model for two interconnected regions 
sharing a common interface. Our main contributions are summarized as follows:
\begin{enumerate}
    \item \textbf{Degenerate reaction-diffusion framework:}  
    We formulate the epidemic dynamics as a coupled system of nonlinear 
    reaction-diffusion partial differential equations with boundary degeneracy \cite{zhan2016reaction,floater1991blow,ge2018compact,einav2020indirect} (PDEs). Degeneracy naturally arises 
    in areas where the population density vanishes, leading to spatial zones with zero mobility (like quarantine, geographic or political barriers).
\item \textbf{Migration-dependent boundary dynamics:}  
Cross-border population movement between the two regions is modelled through 
nonlinear Robin boundary conditions. When the number of infected individuals 
exceeds a certain threshold, mobility restrictions are enforced by switching to 
Neumann boundary conditions to effectively simulate lockdown policies. 
This component of the model extends our earlier work in~\cite{elamraoui2025time}, 
where migration effects were considered in a simpler setting without degenerate diffusion.
    \item \textbf{Well-posedness and numerical analysis:}  
    We prove the existence, uniqueness, and positivity of weak solutions using the 
    Faedo–Galerkin method. The theoretical results are validated through 
    high-resolution numerical simulations based on the Finite Volume Method (FVM).
\end{enumerate}
This modelling framework provides deeper insights into the interplay between migration, 
degenerate diffusion, and policy-driven mobility restrictions, offering a more 
realistic representation of epidemic dynamics in interconnected regions.
\medskip
\\
The remainder of this paper is structured as follows: In Section \ref {section2}, we present our time-space SIR model, which is formulated as a reaction-diffusion system with mixed boundary conditions. In Section \ref {section3}, we demonstrate the model's well-posedness utilizing the Faedo–Galerkin method. Finally, Section \ref {section4} introduces the finite volume method that will be employed in Section  \ref {section5} for numerical examples and simulations to validate and illustrate our results.

	\section{Formulation of a degenerate reaction-diffusion SIR model incorporating immigration}\label{section2}
            For a time horizon $T>0$, let $\Omega\subset \mathbb{R}^d$ (where $d\in \{1,2\}$) be a domain decomposed into two contiguous subregions, $\Omega_1$ and $\Omega_2$, such that $\Omega=\Omega_1 \cup \Omega_2$. The external (unshared) boundary of each subregion is denoted by $\partial\Omega_i$ for $ i\in \{1,\,2\}$, while $\Gamma:=\partial\Omega_1\cap \partial \Omega_2$ represents their shared "virtual" interface.
    \\
	This section is devoted to the development of a model for studying immigration dynamics between the regions of \(\Omega\) across the common interface \(\Gamma\), while incorporating relevant policy considerations. First, we introduce our main assumption and provide its formal mathematical justification.

	\begin{itemize}

	\item \textit{Main assumption:}\\ 
	For $i,j \in \{1,2\}$ with $i \neq j$, we assume that each individual in $\bar\Omega_j$ has a probability $\lambda_j(\cdot,\cdot): [0,T]\times\bar{\Omega}_i\mapsto[0,1]$ of being in $\bar\Omega_i$. This means that we are concerned with the new location of the individual in $\Omega_i$, regardless of its previous location in $\bar\Omega_j$. Also, this accounts for the immigration of individuals into $\bar\Omega_i$ before the disease is detected in that region. Consequently, on one side, we must assume that diffusion occurs for both residents (individuals originally in $\Omega_i$) and immigrants (individuals from $\Omega_j$). On the other side,
 	the probability $\lambda_j$ is defined such that $\lambda_j(t,x)=0$ for all $(t,x)\in [0,T]\times \partial\Omega_i$ (i.e., at the unshared boundaries of $\Omega_i$ and outside $\Omega$), and it maybe used to measure the number of individuals in units of measure (for example using the Poisson process \cite{Privault2013}).
    Conversely, for all $(t,x)\in [0,T]\times \Omega_i$, we have $\lambda_j(t,x)>0$. Similarly, we assume that each resident in $\Omega_i$ has a probability $\lambda_i$ of leaving the region before disease detection is made. Furthermore, the diffusion in region $i$ is directly dependent on this probability $\lambda_i \in [0,1]$. Thus, if $\lambda_i=0$, the diffusion coefficient $\sigma_i(t,x,\lambda_i):=\kappa(t)D(x,\lambda_i)$ vanishes, leading to a degenerate reaction-diffusion system of coupled partial differential equations.
\end{itemize}		
	$\lambda_j > 0$ indicates that individuals originating from $\Omega_j$ (including susceptible, infected, and recovered classes) are present within $\Omega_i$. Specifically, infected immigrants from $\Omega_j$ can introduce new sources of infection. Consequently, susceptible individuals in $\Omega_i$ are subject to infection from two pathways: from local infectious individuals within $\Omega_i$ at rate $\beta_i$, and from immigrants from $\Omega_j$ at rate $\beta_{ij}$.

	Now, we define the notations used throughout the paper:
	\begin{enumerate}
		\item[•] $\beta_i$  is the infection rate in the domain $\Omega_i.$
		\item[•] $\beta_{ij}$ is the infection rate between two  individuals of $\Omega_i$ and $\Omega_j$.
		\item[•] $\gamma_i$  is the recovery rate in the domain $\Omega_i.$
		\item[•] $\lambda_{.,j}$  (resp.  $\lambda_{.,i}$) is the probability  of being in $\Omega_i$ from $\Omega_j$ (resp.  in $\Omega_j$ from $\Omega_i$).
		\item[•] $\alpha^i$ is the migration and immigration rate of individuals across $\Gamma$.
		\item[•] $\Lambda_i$ is the birth rate in each sub-domain $\Omega_i$.
		\item[•] $\mu_S,\mu_I $ and $\mu_R$ are the natural death rates for each compartment.
		\item[•] $N_i$ is the total population in the domain $\Omega_i$.
	\end{enumerate}
	
	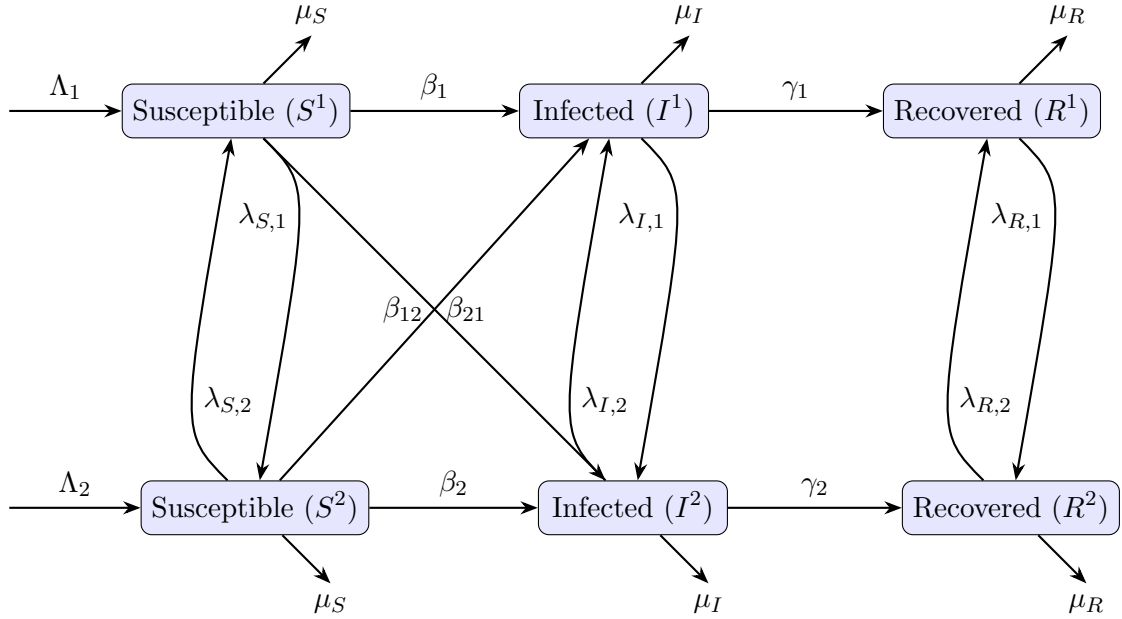
\begin{figure}[H]
		\centering
		
		\begin{tikzpicture}[node distance=5cm, auto, font=\small]
			
			% Define the style for compartments
			\tikzstyle{compartment} = [rectangle, draw, fill=blue!10, text centered, minimum height=.5cm, minimum width=2.5cm, rounded corners]
			
			% Define the style for arrows
			\tikzstyle{arrow} = [thick,->,>=Stealth]
			
			% First SIR compartment (for group 1 or time step 1)
			\node (S1) [compartment] {Susceptible ($S^1$)};
			\node (I1) [compartment, right of=S1] {Infected ($I^1$)};
			\node (R1) [compartment, right of=I1] {Recovered ($R^1$)};
			
			% Arrows for transitions for the first SIR compartment
			\draw [arrow] (S1) -- node[above] {$\beta_1$} (I1);
			\draw [arrow] (I1) -- node[above] {$\gamma_1$} (R1);
			
			% Second SIR compartment (for group 2 or time step 2)
			\node (S2) [compartment, below of=S1,xshift=.25cm,yshift=-0.25cm] {Susceptible ($S^2$)};
			\node (I2) [compartment, right of=S2] {Infected ($I^2$)};
			\node (R2) [compartment, right of=I2] {Recovered ($R^2$)};
			
			% Arrows for transitions for the second SIR compartment
			\draw [arrow] (S2) -- node[above] {$\beta_2$} (I2);
			\draw [arrow] (I2) -- node[above] {$\gamma_2$} (R2);
			
			% Cross-transition between groups with adjusted positioning
			\draw [arrow] (S1) -- node[above,left] {$\beta_{12}$} (I2);
			\draw [arrow] (S2) -- node[below,right] {$\beta_{21}$} (I1);
			
			% Additional cross-transitions
			\draw [arrow] (S2) .. controls +(-1,1) .. node[right]  {$\lambda_{S,2}$} (S1);
			\draw [arrow] (R2) .. controls +(-1,1) .. node[right]  {$\lambda_{R,2}$} (R1);
			\draw [arrow] (S1).. controls +(1,-1) .. node[left]  {$\lambda_{S,1}$} (S2);
			\draw [arrow] (R1).. controls +(1,-1) .. node[left] {$\lambda_{R,1}$} (R2);
			
			% Adjust the position of the arrow from I1 to I2
			\draw [arrow] (I1) .. controls +(1,-1) .. node[left] {$\lambda_{I,1}$} (I2);
			\draw [arrow] (I2) .. controls +(-1,1) .. node[right] {$\lambda_{I,2}$} (I1);
			
			% Birth rates entering S1 and S2
			\draw [arrow] (-3,0) -- node[above] {$  \Lambda_1$} (S1);
			\draw [arrow] (-3,-5.25) -- node[above] {$\Lambda_2$} (S2);
			
			% Death rates leaving S, I, and R compartments
			\draw [arrow] (S1) -- ++(1,1) node[above] {$\mu_S$};
			\draw [arrow] (I1) -- ++(1,1) node[above] {$\mu_I$};
			\draw [arrow] (R1) -- ++(1,1) node[above] {$\mu_R$};
			
			\draw [arrow] (S2) -- ++(1,-1) node[below] {$\mu_S$};
			\draw [arrow] (I2) -- ++(1,-1) node[below] {$\mu_I$};
			\draw [arrow] (R2) -- ++(1,-1) node[below] {$\mu_R$};
			
		\end{tikzpicture}
		\caption{SIR model with cross-transitions, birth, and death rates.}
	\end{figure}

	Considering the assumptions mentioned above, we investigate their influence on the evolution of individuals within domains $\Omega_i$ for $i\in \{1,2\}$. This leads to the following model, defined for $i,j\in\{1,2\}$,  $i\neq j$ and $Q_i:=(0,T)\times \Omega_i$:
	\begin{equation}\label{Pb_main}
		\left\{
		\begin{array}{ll}
			\partial_{t}S^i-\mathrm{div}(\sigma(\lambda_S)\cdot\boldsymbol\nabla S) =\Lambda_i N_i -(\beta_i I^i+\beta_{ij} I^j)S^i+\lambda_{S,j} S^j - (\lambda_{S,i}+\mu_S) S^i,& \text{in } Q_i\\[2ex]
			\partial_{t}I^i-\mathrm{div}(\sigma(\lambda_I)\cdot\boldsymbol\nabla I)=(\beta_i I^i +\beta_{ij} I^j)S^i-\gamma_i I^i+ \lambda_{I,j} I^j -( \lambda_{I,i}+\mu_I) I^i, &\text{in } Q_i \\[2ex]
			\partial_{t}R^i-\mathrm{div}(\sigma(\lambda_R)\cdot\boldsymbol\nabla R)=\gamma_i I^i+ \lambda_{R,j} R^j-(\lambda_{R,i}+\mu_R) R^i &\text{in } Q_i
		\end{array}
		\right.
	\end{equation}
    where
	$\sigma(\lambda_k):=(\sigma_i(\lambda_{k,i}),\sigma_j(\lambda_{k,j}))$  for $i\in\{1,2\}$ and $k\in \{S,\,I,\,R\}$. Moreover, we denoted $\boldsymbol\nabla S^\top:=(\nabla S^i,\nabla S^j)^\top$ and $v\cdot w$ is the inner product.
    
The problem \eqref{Pb_main}  will be associated with the Dirichlet boundary conditions:
	\begin{equation}\label{BC2}
		%\begin{cases}
		S^i=
		I^i =
		R^i=0,\qquad \mbox{on}\;\;[0,T]\times\partial\Omega_i,\text{ for }i\in\{1,2\},%\\
	\end{equation}
	and the following initial conditions:
	\begin{equation}\label{CI}
		\begin{cases}
			S^i(0,x)=S_{0}(x,\lambda_{S,i},\lambda_{S,j})\geq 0,\\
			I^i(0,x)=I_{0}(x,\lambda_{I,i},\lambda_{I,j})\geq 0,\\
			R^i(0,x)=R_{0}(x,\lambda_{R,i},\lambda_{R,j})\geq 0,
		\end{cases}
	\end{equation}
	meaning that the initial number of individuals in the considered region $\Omega_i$ depends on the probability of leaving or being in this region. So, essentially, we are dealing with a semi-linear boundary weakly degenerating parabolic problem.

   We assume the following constraint on the  virtual interface $\Gamma$:
	\begin{equation}\label{BC1}
		\begin{cases}
			\sigma(\lambda_S)\cdot\boldsymbol\nabla S \cdot   \tau_i+\alpha^j_S(I^i) S^j=0\\
			\sigma(\lambda_I)\cdot\boldsymbol\nabla I \cdot   \tau_i+\alpha_I^j(I^i) I^j=0\\
			\sigma(\lambda_R)\cdot\boldsymbol\nabla R \cdot   \tau_i+\alpha^j_R(I^i) R^j=0
		\end{cases}\qquad  \mbox{on}\;\;[0,T]\times\Gamma.
	\end{equation}
	%and the diffusion satisfies the following properties: 

	  The notation $\tau_i $ (respectively $\eta_i)$) stands for outward boundary normal vectors from $\Omega_i$ on the  boundary $\Gamma$ (respectively  $\partial\Omega_{i}$).For simplicity of the notation, let us set $\tau=\tau_1=-\tau_2$. \\
The function $\alpha_k^j(I)$ models the direction of the movement between two regions and is defined for $k \in \{S, I, R\}$ as follows:
\begin{equation*}
\alpha_k^j(I) := 
\begin{cases}
\;\;\;\alpha(I), & \text{if } I \geq I_{\mathrm{th}}^j, \\[6pt]
-\alpha(I), & \text{otherwise},
\end{cases}
\end{equation*}
where $I_{\mathrm{th}}^j$ is a threshold value specific to the region $\Omega_j$.\\
We suppose $\alpha(\cdot)$ to be a Lipschitz continuous function such that 
\[
\alpha(n) \to 0 \quad \text{as} \quad n \to +\infty.
\]
Here are a few typical examples that can be considered is:
\[
\alpha(I) = \frac{1}{1 + I^2}, \quad \alpha(I) = e^{-I}.
\]
In real-world applications, $\alpha(I)$ represents the mobility or transfer rate of individuals between regions, and it decreases as the number of infected individuals increases. This reflects the realistic assumption that higher infection levels reduce movement due to health precautions or public policies.\\
 When the number of infected individuals in region $\Omega_j$ exceeds the threshold $I_{\mathrm{th}}^j$, individuals are assumed to leave region $\Omega_i$ and move toward $\Omega_j$, with mobility rate $\alpha(I)$. Conversely, if the infection level in $\Omega_j$ is below the threshold, individuals may migrate in the opposite direction, from $\Omega_j$ to $\Omega_i $, with mobility rate $-\alpha(I)$.\\
In extreme cases, such as during a pandemic, a lockdown policy may be enforced, corresponding to $\alpha(I) = 0$. This scenario models a complete halt of mobility between regions. The condition $\alpha(I) \to 0$ as $I \to +\infty$ captures this behavior, reflecting that as infection levels become very large, mobility naturally tends to zero due to restrictive public health measures.

    \begin{remark}
        The nonlinear virtual interface conditions \eqref{BC1} are prescribed on the interface between two subregions of $\Omega$. They not only serve to model the lockdown but also act as coupling constraints, ensuring the correct transmission of information across the subdomains of $\Omega$.
    \end{remark}
	In the next section, we aim to prove the well-posedness of the problem stated above. For this purpose, we will need the following assumptions:  

    \begin{equation*}%\label{siginf}
		\begin{aligned}
			&\sigma_i \in \mathcal C^0(\overline{Q_i})\cap \mathcal{C}^1(Q_i) \text{ and is positive in } [0,T]\times \Omega_i,\\
			&\frac{\partial_t\sigma_i}{\sigma_i} \in L^{\infty}(Q_i).
		\end{aligned}
	\end{equation*}
    and there exists $0 <\delta < \min\{t, T-t\}$
	such that
	\begin{equation}\label{wdc}
		\int_{t-\delta}^{t+\delta} \int_{\Omega_i\cap B_\delta(x)} \frac{1}{\sigma_i(\lambda_i(y,s))} \mathrm{d} y \mathrm{~d} s<+\infty\qquad\forall t\in (0,T),
	\end{equation}
	with $B_\delta(x)$ being the ball in $\mathbb{R}^d$ centered at $x$ and with radius $\delta$.
	Consequently, our analysis will focus on the case of weak degeneracy (see \cite{Wang2010}).
	
	\section{Well-posedness of the problem}\label{section3}
	
	In this section, we prove the well-posedness of the initial boundary value problem \eqref{Pb_main}-\eqref{CI}. To this end, we introduce the following notations
	\begin{equation*}
		u^i=\begin{pmatrix}
			S ^i   \\[1ex]
			I^i \\[1ex]
			R^i
		\end{pmatrix},\; \boldsymbol\nabla u^i= \begin{pmatrix}
			\nabla S^i   \\[1ex]
			\nabla I^i \\[1ex]
			\nabla R^i
		\end{pmatrix},
	\end{equation*}
	\begin{equation*}
		f(u^i,u^j)=\begin{pmatrix}
			\Lambda_i N_i-\mu_S S^i   -(\beta_i I^i+\beta_{ij}I^j) S^i    \\[1ex]
			(\beta_i I^i+\beta_{ij}I^j) S^i-(\gamma_i+\mu_I) I^i
			\\[1ex]
			\gamma_i I^i-\mu_R R^i
		\end{pmatrix},\,
		\lambda_i= \begin{pmatrix}
			\lambda_{ S,i}   \\[1ex]
			\lambda_{ I,i}\\[1ex]
			\lambda_{ R,i}
		\end{pmatrix},
	\end{equation*}
	$A_i(\lambda_i):=\operatorname{diag}\left(\sigma_{S,i},\sigma_{I,i},\sigma_{R,i}\right)$, where $\sigma_{\cdot,i}=\sigma_i(\lambda_{\cdot,i})$,    
	and  
	$\alpha^j=\operatorname{diag}\left(\alpha_S ^j(I^i),\alpha_I^j(I^i),\alpha_R^j(I^i)\right)$.

	Then the problem \eqref{Pb_main}-\eqref{CI} becomes:
	
	\begin{equation}\label{Pb_main1}
		\left\{
		\begin{array}{ll}
			\partial_{t}u^i-\mathrm{div}(A_i(\lambda_i)\boldsymbol\nabla u^i) =f(u^i,u^j)+\mathrm{div}(A_j(\lambda_j)\boldsymbol\nabla u^j)+\lambda_j   u^j - \lambda_i   u^i,& \text{in } Q_i\\[1ex]
			A_i(\lambda_i)\boldsymbol\nabla u^i \cdot   \tau+A_j(\lambda_j)\boldsymbol\nabla u^j \cdot   \tau+\alpha^j(I^i) u^j=\boldsymbol 0,&\mbox{on}\;\;[0,T]\times\Gamma\\[1ex]
			u^i=\boldsymbol 0&\mbox{on}\;\;[0,T]\times\partial\Omega_i\\[1ex]
			u^i(0,x)=u^i_0(x,\lambda_i,\lambda_j)&  \text{in } \Omega_i\\[1ex]  \text{ for  all } i,j=1,2 \text{ and } i\neq j.
		\end{array}
		\right.
	\end{equation}

	Before going further, let us first introduce some weighted Sobolev spaces that are naturally associated with degenerate operators:
	\begin{align*}
		\boldsymbol H_{A_i}^1 (\Omega):= \Big\{ u \in \boldsymbol L^2(\Omega):=L^2(\Omega)^3: \nabla u\cdot A_i \nabla u\in \boldsymbol L^1(\Omega):=L^1(\Omega)^3 \Big\}
	\end{align*}
	endowed with the norm 
	$$
	\| u \|_{\boldsymbol H_{A_i}^1}^2:= \int_{\Omega}|u|^2 + \int_{\Omega}|A_i^{\tfrac{1}{2}} \nabla u|^2,\quad \forall u \in \boldsymbol H_{A_i}^1 (\Omega).
	$$
	It is easy to observe that the classical Sobolev space $\boldsymbol H^1(\Omega)\subset  \boldsymbol H_{A_i}^1 (\Omega)$. We also denote
	$$\boldsymbol H_{A_i}^2 (\Omega)= \Big\{ u \in \boldsymbol H_{A_i}^1(\Omega): \operatorname{div}(A_i \nabla u) \in  \boldsymbol L^2(\Omega)\Big\}$$
	with the associated norm defined by
	\begin{align*}
		\| u \|_{\boldsymbol H_{A_i}^2}^2 :&=  \| u \|_{\boldsymbol H_{A_i}^1}^2 +  \| \mathrm{div}(A_i\boldsymbol \nabla u ) \|_{\boldsymbol L^2}^2\\
		&= \int_{\Omega_i}|u|^2 + \int_{\Omega}|A_i^{\tfrac{1}{2}} \nabla u|^2 + \int_{\Omega}|\operatorname{div}(A_i \nabla u)|^2,\quad \forall u \in \boldsymbol H_{A_i}^2 (\Omega).
	\end{align*}
	\begin{remark}
		It is worth noting that, for these natural norms, $\boldsymbol H_{A_i}^1(\Omega)$, $\boldsymbol H_{A_i}^2(\Omega)$ are Hilbert spaces. Moreover, we have the following continuous
		embeddings are satisfied
		$$
		\boldsymbol H^1(\Omega)  \hookrightarrow \boldsymbol H_{A_i}^1(\Omega) \hookrightarrow \boldsymbol L^2(\Omega);\quad \boldsymbol H_{A_i}^2(\Omega) \hookrightarrow	\boldsymbol H_{A_i}^1(\Omega).
		$$
		Addionally, we have that $ \boldsymbol H_{A_i}^1(\Omega) \subset \boldsymbol H^1_{\text{loc}}(\Omega)$ and $ \boldsymbol H_{A_i}^2(\Omega) \subset \boldsymbol H^2_{\text{loc}}(\Omega)$.
	\end{remark}
	In view of \cite[Lemma 1]{Araruna}, we have $\boldsymbol C^\infty(\overline{\Omega})$ is dense in $\boldsymbol H_{A_i}^1(\Omega)$. This permits to define 
	$$
	\boldsymbol H_{A_i,0}^1(\Omega_i):=\overline{\left\{u\in \boldsymbol C^\infty(\overline{\Omega}) \mid \operatorname{supp}(u)\subset\subset \Omega \right\}}^{\boldsymbol H_{A_i}^1(\Omega)}.
	$$
	Furthermore, the space $\boldsymbol H_{A_i,0}^1(\Omega)$ can be endowed with the norm
	$$
	|u|^2_{A_i,0}:= \int_{\Omega}|A_i^{\tfrac{1}{2}} \nabla u|^2,\quad \forall u \in \boldsymbol H_{A_i,0}^1 (\Omega),
	$$
	which is equivalent to the norm $\|\cdot\|_{\boldsymbol H_{A_i,0}^1(\Omega_i)}$ in $\boldsymbol H_{A_i,0}^1(\Omega_i)$.
	
	The following compactness results hold (for the proofs, see \cite[
	Theorem 3.4]{Wu2025}:% or \cite[Appendix]{Alabau2006}): 
	\begin{proposition}
		One has
		\begin{enumerate}
			\item The space $\boldsymbol H_{A_i,0}^1(\Omega)$ is compactly imbedded in $\boldsymbol L^2(\Omega)$.
			%\item The space $\boldsymbol H_{A_i,0}^2(\Omega)$ is compactly imbedded in $\boldsymbol H_{A_i,0}^1(\Omega)$.
		\end{enumerate}
	\end{proposition}
	 
	After recalling the definition and key results concerning the weighted Sobolev space, along with some of its important continuity and compact embedding properties known in the literature, we now define the weak solution of the system \eqref{Pb_main1}.
	\begin{definition}
		We say that a function \( u = (u^1, u^2) \)  such that
		$$
		(u^1, u^2) \in \boldsymbol{L}^2(0, T; \boldsymbol{H}^1_{A_1,0}(\Omega)) \times \boldsymbol{L}^2(0, T; \boldsymbol{H}^1_{A_2,0}(\Omega)) 
		$$ 
		and $$
		(\partial_t u^1, \partial_t u^2) \in \boldsymbol{L}^2(0, T;\boldsymbol{H}^{-1}_{A_1,0}(\Omega)) \times \boldsymbol{L}^2(0, T;\boldsymbol{H}^{-1}_{A_2,0}(\Omega))
		$$
		is a weak solution of the degenerate initial/boundary value problem \eqref{Pb_main1} provided for almost every \( t \in (0, T) \) and for all test functions \( (v^1,v^2) \in \boldsymbol{L}^2(0, T;\boldsymbol{H}^1_{A_1,0}(\Omega) \times \boldsymbol{H}^1_{A_2,0}(\Omega)) \):
		\begin{equation}\label{FV0}
			\begin{cases}
				\int_{\Omega_i} \partial_t u^i v^i + \int_{\Omega_i} A_i(\lambda_i) \nabla u^i \cdot \nabla v^i + \int_{\Omega_i} A_j(\lambda_j) \nabla u^j \cdot \nabla v^i = \int_{\Omega_i} f(u^i, u^j) v^i +\\ \qquad \qquad \int_{\Gamma} \alpha^j(u^i) u^j v^i  \qquad + \int_{\Omega_i} \lambda_j u^j v^i - \int_{\Omega_i} \lambda_i u^i v^i, \\
				\text{for all } i, j = 1, 2 \text{ and } i \neq j.
			\end{cases}
		\end{equation}
		%Moreover, \( u \) satisfies the initial condition:
		and
		\[
		u^i(0, x) = u_0^i(x) \quad \text{in } \Omega_i, \quad \text{for } i = 1, 2,
		\]
		where \( u_0^i \in \boldsymbol L^2(\Omega) \) are given initial data.
	\end{definition}
	Next, we state our main result on the well-posedness of the degenerate problem \eqref{Pb_main1}.
	\begin{theorem}\label{thoe1}
		Let $T>0$ and $u^i_{0}\in\boldsymbol L^2(\Omega), \; i=1,2$, then the problem \eqref{Pb_main1} admits a unique non-negative weak solution $\ds(u^1,u^2)\in \boldsymbol L^{\infty}\big(0,T,\boldsymbol L^{2}(\Omega)\big)^2 \cap \boldsymbol L^2(0,T,\boldsymbol H^1_{A_1,0}(\Omega))\times \boldsymbol L^2(0,T,\boldsymbol H^1_{A_2,0}(\Omega))$ and  $(\partial_{t}u^1,\partial_{t}u^2)\in\boldsymbol L^2(0,T,\boldsymbol H^{-1}_{A_1,0}(\Omega))\times \boldsymbol L^2(0,T,\boldsymbol H^{-1}_{A_2,0}(\Omega))$.
	\end{theorem}
	The proof of Theorem \ref{thoe1} relies on the classical Galerkin method and a truncation technique. A first step in this process is establishing the well-posedness of an intermediate problem. Let $ g :\boldsymbol H^1_{A_1}(\Omega)\times \boldsymbol H^1_{A_2}(\Omega)\rightarrow \boldsymbol L^{2}(\Omega) $ be a globally Lipschitz function with a Lipschitz constant  $L_g>0$, satisfying $g(0,0)=0$. We consider the following problem
	\begin{equation}\label{abrP}
		\left\{
		\begin{array}{ll}
			\partial_{t}u^i-\mathrm{div}(A_i(\lambda_i)\boldsymbol\nabla u^i) =g(u^i,u^j)+\mathrm{div}(A_j(\lambda_j)\boldsymbol\nabla u^j)+\lambda_j   u^j - \lambda_i   u^i,& \text{in } Q_i\\[1ex]
			A_i(\lambda_i)\boldsymbol\nabla u^i \cdot   \tau+A_j(\lambda_j)\boldsymbol\nabla u^j \cdot   \tau+\alpha^j(I^i) u^j=\boldsymbol 0,&\mbox{on}\;\;[0,T]\times\Gamma\\[1ex] 
			u^i=\boldsymbol 0&\mbox{on}\;\;[0,T]\times\partial\Omega_i\\[1ex]
			u^i(0,x)=u^i_0(x,\lambda_i,\lambda_j)&  \text{in } \Omega_i\\[1ex]  \text{ for  all } i,j=1,2 \text{ and } i\neq j.
		\end{array}
		\right.
	\end{equation}
	Let us define the auxiliary operator
\begin{equation}\label{Zoperator}
\begin{split}  
\mathcal{I}: L^2(\Omega_1)\times L^2(\Omega_2) &\to  L^2(\Omega_1)\times L^2(\Omega_2),\\
(v_1,v_2) &\mapsto  \mathcal{I}(v_1,v_2)=(\lambda_2\, v_2,\;\lambda_1\, v_1),
\end{split}
\end{equation} 
and recall that \(\lambda_i \in [0,1]\) represents the probability that an individual migrates from \(\Omega_i\) to the other domain.  
Under the main assumption, \(\mathcal{I}\) is well-defined: each individual can immigrate, so that \(v_i\) may contribute to \(\Omega_j\) whenever the probability of immigration from \(\Omega_i\) to \(\Omega_j\) is positive, i.e. \(\lambda_i>0\) for \(i\neq j\) which in such case means that $v_i$ is well defined on the domain $\Omega_j$ by its new location.  
 In addition, it is straightforward to show that 
	this operator is linear and bounded, i.e. $\exists M_{\mathcal I}>0$ such that 
	\begin{equation}
		\Vert \mathcal I(v_1,v_2)\Vert_{\boldsymbol L^2(\Omega_1)\times \boldsymbol L^2(\Omega_2)} \leq M_{\mathcal I}\Vert(v_1,v_2)\Vert_{\boldsymbol L^2(\Omega_1)\times \boldsymbol L^2(\Omega_2)}.
	\end{equation}
	We now establish the well-posedness of the intermediate problem \eqref{abrP}.
	\begin{proposition}\label{prop3}
		Let $T>0$ and $u^i_{0}\in\boldsymbol L^2(\Omega), \; i=1,2$. The problem \eqref{abrP} admits a unique non-negative weak solution $\ds(u^1,u^2)\in \boldsymbol L^{\infty}\big(0,T,\boldsymbol L^{2}(\Omega)\big)^2 \cap \boldsymbol L^2(0,T,\boldsymbol H^1_{A_1,0}(\Omega))\times \boldsymbol L^2(0,T,\boldsymbol H^1_{A_2,0}(\Omega))$ and  $\ds(\partial_{t}u^1,\partial_{t}u^2)\in \boldsymbol L^2(0,T,\boldsymbol H^{-1}_{A_1,0}(\Omega))\times \boldsymbol L^2(0,T,\boldsymbol H^{-1}_{A_2,0}(\Omega))$.
	\end{proposition}
	\begin{proof}
		%[Proof of the proposition \eqref{prop3}]
		The proof is based on the Faedo-Galerkin procedure. \\
		
		\textbf{Step 1: Existence in finite dimension}\\
		
		The space $\boldsymbol L^2(\Omega)$ is a separable space, hence there exists an increasing sequence of subspace $\boldsymbol V_n\subset \boldsymbol V_{n+1}$ which are dense in $\boldsymbol L^2(\Omega)$ where
		for every
		$n \geq 0$  the finite dimensional space $\boldsymbol V_n$ is defined by
		$$ \boldsymbol V_n = span\left\lbrace v_1,v_2, ...,v_n \right\rbrace $$
		such that  $\left\lbrace v_1,v_2, ...,v_n.... \right\rbrace$  is a complete orthonormal basis of $\boldsymbol L^2(\Omega_i)$ i.e., $\Vert v_i\Vert_{\boldsymbol L^2(\Omega)} =1$ and $ \langle v_i,v_j \rangle_{\boldsymbol L^2(\Omega_i)}=0, \forall i\neq j$.

		We consider the following finite problem corresponding to \eqref{abrP} on $\boldsymbol V_{n}$ :
		\begin{equation}\label{abrPN}
			\left\{
			\begin{array}{ll}
				\partial_{t}  u_n^i-\mathrm{div}(A_i(\lambda_i)\boldsymbol\nabla u^i_n) = g(  u_n^i,u^j_n)+\mathrm{div}(A_j(\lambda_j)\boldsymbol\nabla u^j_n)+\lambda_j   u_n^j - \lambda_i   u_n^i, &  \text{in } Q_i \\[1ex]
	
				A_i(\lambda_i)\boldsymbol\nabla u^i_n \cdot   \tau+A_j(\lambda_j)\boldsymbol\nabla u^j_n \cdot   \tau+\alpha^j(I^i) u^j_n=\boldsymbol 0, &\text{on  }[0,T]\times\Gamma  \  \\[1ex]
                u_n^i =\boldsymbol 0, & \text{on  }[0,T]\times \partial\Omega_{i}  \  \\[1ex]
				u_n^i(0,x)=  u^i_{0}(x,\lambda_i,\lambda_j),  &  \text{in } \Omega_i\\[1ex]  \text{ for  all } i,j=1,2 \text{ and } i\neq j.&
			\end{array}
			\right.
		\end{equation}
		We look for approximate solutions $(u^i_{n})$ for the problem $\eqref{abrPN}$ in the form
		\begin{equation}\label{defUn}
			u^i_{n}(t,x)=\sum_{k=1}^{n}  d^i_{k}(t)v_{k},\quad \text{for }i=1,2.
		\end{equation}
		Multiplying equation \eqref{abrPN} with function test $v_k$ formulates the following Galerkin variational formulation for our problem:
		%\begin{equation}\label{FV}
		\begin{equation}\label{FV}
			\left\{
			\begin{array}{ll}
				\int_{\Omega_i}  \partial_{t}  u_n^i  v_k  +\int_{\Omega_i}A_i(\lambda_i) \boldsymbol\nabla    u_n^i  \boldsymbol\nabla v_k+&     \int_{\Omega_i}A_j(\lambda_j) \boldsymbol\nabla    u_n^j  \boldsymbol\nabla v_k =\int_{\Omega_i} g   (  u_n^i ,u_n^j) v_k \\& +\int_{\Gamma}    \alpha^j(I^i_n) u_n^j  v_k+\int_{\Omega_i} \lambda_j   u_n^j v_k-\int_{\Omega_i} \lambda_i   u_n^i  v_k,   \\[2ex]
				\text{ for  all } i,j=1,2 \text{ and } i\neq j.&
			\end{array}
			\right.
		\end{equation}

		Substituting \eqref{defUn} into \eqref{FV} leads to the ODE system:     for  all $ i,j=1,2$  and $i\neq j$ we have
		\begin{equation}\label{ODEsystem}
			\begin{split}
				\mathbf{M}_i  (d^i_k)^{'} + \mathbf{K}_i d^i_k + \mathbf{K}_j d^j_k = \mathbf{G}(d^i_k, d^j_k) + \lambda_j \mathbf{M}_i d^j_k - \lambda_i \mathbf{M}_i d^i_k\\
				\text{ for  all } i,j=1,2, i\neq j  \text{ and } k=1,2\dots, n.
			\end{split}
		\end{equation}
  where \(\mathbf{M}_i\) denotes the mass matrix, \(\mathbf{K}_i\) the stiffness matrix, and \(\mathbf{G}\) represents the nonlinear reaction component. It may be useful to explicitly define these terms as follows:
\begin{equation*}
\begin{aligned}
[\mathbf{M}_i]_{kl} &= \int_{\Omega_i} v_l \, v_k \, dx, \\
[\mathbf{K}_i]_{kl} &= \int_{\Omega_i} A_i(\lambda_i) \nabla v_l \cdot \nabla v_k \, dx, \\
[\mathbf{G}(d^i, d^j)]_k &= \int_{\Omega_i} g\left(\sum_{l=1}^n d^i_l v_l, \sum_{m=1}^n d^j_m v_m\right) v_k \, dx 
+ \int_{\Gamma} \alpha^j(I_n^i) \left(\sum_{m=1}^n d^j_m v_m\right) v_k \, dy.
\end{aligned}
\end{equation*}
\\
		This coupled  differential system can be completed with the initial conditions
		\begin{equation}\label{ODEsystemIC}
			d^{i}_k(0)=\langle u_{0},v_{k}\rangle_{L^2(\Omega)},\;\;\mbox{for all }\;i=1,2\text{ and }k=1...,n,
		\end{equation}
		Thanks to the global Lipschitz condition on $g$ and the boundedness of $\alpha^j$, the standard theory of ordinary differential equations provides the local existence and uniqueness of solutions $\bigg(d^{i}_1(t),d^{i}_2(t),\dots, d^{i}_n(t)\bigg)$, where $i=1,2$ for the differential system \eqref{ODEsystem}-\eqref{ODEsystemIC}.
		Hence,  there exists a unique local solution $(u^1_{n},u^2_{n})$ of \eqref{FV}.

		\textbf{Step 2: Estimation of $ (u_n^1,u_n^2)$}: \\
		The next step is concerned with establishing energy estimates for the solution $u_{n}$ and its derivatives. Starting by taking $ v_k= u_n^1 $ for $i=1$ and $ v_k= u_n^2 $ for $i=2$ in equation \eqref{FV}, we obtain:
		\begin{equation}\label{FV1}
			\begin{split}
				\int_{\Omega_1}  \partial_{t}  u_n^1  u_n^1  +\int_{\Omega_1}A_1(\lambda_1) \boldsymbol\nabla    u_n^1  \boldsymbol\nabla u_n^1 +&\int_{\Omega_1}A_2(\lambda_2) \boldsymbol\nabla    u_n^2  \boldsymbol\nabla u_n^1 =\int_{\Omega_1}    g(  u_n^1 ,u_n^2) u_n^1  +\\&
                \int_{\Gamma}    \alpha(I^1_n) u_n^1  u_n^1+\int_{\Omega_1} \lambda_2   u_n^2 u_n^1 -\int_{\Omega_1} \lambda_1   u_n^1  u_n^1 ,   \\
				\int_{\Omega_2}  \partial_{t}  u_n^2  u_n^2  +\int_{\Omega_2}A_2(\lambda_2) \boldsymbol\nabla    u_n^2  \boldsymbol\nabla u_n^2+&\int_{\Omega_2}A_1(\lambda_1) \boldsymbol\nabla    u_n^1 \boldsymbol\nabla u_n^2 =\int_{\Omega_2}    g(  u_n^2 ,u_n^1) u_n^2 +\\&   \int_{\Gamma}    \alpha(I^2_n) u_n^2  u_n^2+\int_{\Omega_2} \lambda_1  u_n^1 u_n^2-\int_{\Omega_2} \lambda_2   u_n^2  u_n^2.
			\end{split}
		\end{equation}
		Summing the two equations, and using Young's inequality (given that $\lambda_i \in [0,1]$), yields:
		\begin{equation*}
			\begin{split}
				\int_{\Omega_1}  \partial_{t}  u_n^1   u_n^1 +\int_{\Omega_2}  \partial_{t}  u_n^2   u_n^2+ \Vert   u_n^1 \Vert^2_{\boldsymbol H^1_{A_1,0}(\Omega_1)} +\Vert   u_n^2 \Vert^2_{\boldsymbol H^1_{A_2, 0}}(\Omega_2)\leq   \int_{\Omega_2}  \ds  g(  u_n^2,u_n^1) u_n^2+ \Vert  u_n^1 \Vert^2_{\boldsymbol L^2(\Omega_1)}+\\  \Vert   u_n^2 \Vert^2_{\boldsymbol L^2(\Omega_2)}+
				\int_{\Omega_1}  \ds  g(  u_n^1,u_n^2) u_n^1+  \frac{ \epsilon}{2} \Vert \lambda_1 u_n^1  \Vert^2_{\boldsymbol L^2 (\Omega_2)} +\frac{ 1}{2\epsilon} \Vert u_n^2  \Vert^2_{\boldsymbol L^2(\Omega_2)}+\frac{1}{2\epsilon } \Vert u_n^1  \Vert^2_{\boldsymbol L^2(\Omega_1)}\\+ \frac{ \epsilon}{2} \Vert \lambda_2 u_n^2  \Vert^2_{\boldsymbol L^2 (\Omega_1)}-\int_{\Omega_1}A_2(\lambda_2) \boldsymbol\nabla    u_n^2  \boldsymbol\nabla u_n^1-\int_{\Omega_2}A_1(\lambda_1) \boldsymbol\nabla    u_n^1 \boldsymbol\nabla u_n^2.\\
   \int_{\Gamma}    \alpha(I^1_n) (u_n^1)^2+\alpha(I^2_n) (u_n^2)^2	
			\end{split}
		\end{equation*}
		Let us estimate the terms $\int_{\Omega_1}A_2(\lambda_2) \boldsymbol\nabla    u_n^2  \boldsymbol\nabla u_n^1$ and $\int_{\Omega_2}A_1(\lambda_1) \boldsymbol\nabla  u_n^1 \boldsymbol\nabla u_n^2$. By the Young inequality with $\epsilon>0$, we have
		\begin{align*}
			(\star)\quad\vert \int_{\Omega_1}A_2(\lambda_2) \boldsymbol\nabla    u_n^2  \boldsymbol\nabla u_n^1\vert
			%&=\vert \int_{\Omega_1}\sqrt{A_2(\lambda_2)} \boldsymbol\nabla    u_n^2  \sqrt{A_2(\lambda_2)}\boldsymbol\nabla u_n^1\vert\\&
			\leq \frac{\epsilon}{2} \int_{\Omega_1}\vert \sqrt{A_2(\lambda_2)} \boldsymbol\nabla    u_n^2\vert^2  +\frac{1}{2\epsilon}\int_{\Omega_1} \vert \sqrt{A_2(\lambda_2)}\boldsymbol\nabla u_n^1\vert ^2
		\end{align*}
		and similarly
		\begin{align*}
			(\star \star)\quad\vert \int_{\Omega_2}A_1(\lambda_1) \boldsymbol\nabla    u_n^1  \boldsymbol\nabla u_n^2\vert
			%&=\vert \int_{\Omega_2}\sqrt{A_1(\lambda_1)} \boldsymbol\nabla    u_n^1  \sqrt{A_1(\lambda_1)}\boldsymbol\nabla u_n^2\vert\\&
			\leq \frac{\epsilon}{2} \int_{\Omega_2}\vert \sqrt{A_1(\lambda_1)} \boldsymbol\nabla    u_n^1\vert^2  +\frac{1}{2\epsilon}\int_{\Omega_2} \vert \sqrt{A_1(\lambda_1)}\boldsymbol\nabla u_n^2\vert ^2
		\end{align*}
		By  $(\star )$ and $(\star\star )$, we obtain: 
		\begin{equation*}
			\begin{split}
				\int_{\Omega_1}  \partial_{t}  u_n^1   u_n^1 +\int_{\Omega_2}  \partial_{t}  u_n^2   u_n^2+(1-\frac{ 1}{2 \epsilon}) \Vert   u_n^1 \Vert^2_{\boldsymbol H^1_{A_1,0}(\Omega_1)} +(1-\frac{ 1}{2 \epsilon}) \Vert   u_n^2 \Vert^2_{\boldsymbol H^1_{A_2,0}(\Omega_2)}\leq 
				\Vert   u_n^1 \Vert^2_{\boldsymbol L^2(\Omega_1)}+ \\ \Vert   u_n^2 \Vert^2_{\boldsymbol L^2(\Omega_2)}+
				\Vert \ds  g(  u_n^2,u_n^1)\Vert_{\boldsymbol L^2(\Omega_2)} \Vert   u_n^2\Vert_{\boldsymbol L^2(\Omega_2)}+\Vert \ds  g(  u_n^1,u_n^2)\Vert_{\boldsymbol L^2(\Omega_1)} \Vert   u_n^1\Vert_{\boldsymbol L^2(\Omega_1)}+\\   \frac{\epsilon}{2}   \Vert \lambda_1 u_n^1  \Vert^2_{\boldsymbol L^2(\Omega_2)}  + \frac{\epsilon}{2}   \Vert \lambda_2  u_n^2  \Vert^2_{\boldsymbol L^2(\Omega_1)}+   C_{\Gamma }\Vert u_n^1\Vert_{L^2(\Omega_1)}^2+ C_{\Gamma }\Vert u_n^2\Vert_{L^2(\Omega_2)}^2	
			\end{split}
		\end{equation*}
 $ C_ { \Gamma } >0 $ is derived from the trace theorems.\\	

		Since $g$ is a globally Lipschitz function satisfying $g(0,0)=0$, and by the definition and properties of the operator $\mathcal{I}$, we obtain:
		\begin{equation*}
			\begin{split}
				\int_{\Omega_1}  \partial_{t}  u_n^1   u_n^1 +\int_{\Omega_2}  \partial_{t}  u_n^2   u_n^2+ \bigg(1-\frac{ 1}{2 \epsilon}\bigg) \Vert   u_n^1 \Vert^2_{\boldsymbol H^1_{A_1,0}(\Omega_1)} + \bigg(1-\frac{ 1}{ 2 \epsilon}\bigg) \Vert   u_n^2 \Vert^2_{\boldsymbol H^1_{A_2,0}(\Omega_2)}\leq \\   \bigg(1+L_g+ \frac{\epsilon}{2} M_{\mathcal{I}}^2+C_ { \Gamma } \bigg)  \Vert (u_n^1,u_n^2)  \Vert^2_{\boldsymbol L^2(\Omega_1)\times \boldsymbol L^2(\Omega_2) }.
			\end{split}
		\end{equation*}
		Then 
		\begin{equation*}
			\begin{split}
				\int_{\Omega_1}  \partial_{t}  u_n^1   u_n^1 +\int_{\Omega_2}  \partial_{t}  u_n^2   u_n^2+ \bigg(1-\frac{ 1}{2 \epsilon}\bigg)\Vert ( u_n^1,u_n^2) \Vert^2_{\boldsymbol H^1_{A_1,0}(\Omega_1)\times \boldsymbol H^1_{A_2,0}(\Omega_2) } \leq \\  \bigg(1+L_g+ \frac{\epsilon}{2} M_{\mathcal{I}}^2+C_ { \Gamma } \bigg)  \Vert (u_n^1,u_n^2)  \Vert^2_{\boldsymbol L^2(\Omega_1)\times \boldsymbol L^2(\Omega_2) }.
			\end{split}
		\end{equation*}
		 Choosing $\epsilon >0$ so that $ \epsilon >\frac{ 1}{2}$, we have
		\begin{equation}\label{eqb1}
			\begin{small}
				\frac{1}{2} \frac{d }{d t}\Vert ( u_n^1(.,s),u_n^2(.,s)) \Vert^2_{\boldsymbol L^2(\Omega_1)\times \boldsymbol L^2(\Omega_2)}+\Vert ( u_n^1,u_n^2) \Vert^2_{\boldsymbol H^1_{A_1,0}(\Omega_1)\times\boldsymbol  H^1_{A_2,0}(\Omega_2) } \leq  M_2  \Vert (u_n^1,u_n^2)  \Vert^2_{\boldsymbol L^2(\Omega_1)\times \boldsymbol L^2(\Omega_2) }
			\end{small}
		\end{equation}
		and thus 
		\begin{equation*}
			\begin{split}
				\frac{1}{2} \frac{d }{d t}\Vert ( u_n^1(.,s),u_n^2(.,s)) \Vert^2_{\boldsymbol L^2(\Omega_1)\times \boldsymbol L^2(\Omega_2)}\leq  M_2 \Vert (u_n^1,u_n^2)  \Vert^2_{\boldsymbol L^2(\Omega_1)\times \boldsymbol L^2(\Omega_2) }.
			\end{split}
		\end{equation*}
		%Gronwall's lemma yields there exist $\tilde{C}>0$  such that 
		Gronwall's lemma yields the existence of a constant $\tilde{C}>0$ such that:
		\begin{equation}\label{estimationUn}
			\begin{split}
				\Vert ( u_n^1(.,t),u_n^2(.,t)) \Vert^2_{\boldsymbol L^2(\Omega_1)\times \boldsymbol L^2(\Omega_2)}\leq \tilde{C}:=e^{2M_2 t}\|u^1_n(.,0),u^2_n(.,0)\|^2_{\boldsymbol L^2(\Omega_1)\times \boldsymbol L^2(\Omega_2)}
			\end{split}
		\end{equation}
		
		we can get the following estimation 
		\begin{equation}\label{estimationUn0}
			\begin{split}
				\max_{0\leq t \leq T}\Vert ( u_n^1(.,t),u_n^2(.,t)) \Vert^2_{\boldsymbol L^2(\Omega_1)\times \boldsymbol L^2(\Omega_2)}\leq \tilde{C}:=e^{2M_2 t}\|u^1_n(.,0),u^2_n(.,0)\|^2_{\boldsymbol L^2(\Omega_1)\times \boldsymbol L^2(\Omega_2)}
			\end{split}
		\end{equation}
		Returning once more to inequality \eqref{eqb1}, we integrate 0 and $T$ and the employ the ineguality above to find
		\begin{equation}\label{estUn2}
			\begin{split}
				\Vert  (u_n^1,u_n^2 )\Vert^2_{\boldsymbol L^2(0,T,\boldsymbol H^1_{A_1,0}(\Omega_1))\times \boldsymbol L^2(0,T,\boldsymbol H^1_{A_2,0}(\Omega_2))}=\int_0^T \Vert u_n^1,u_n^2 \Vert^2_{\boldsymbol H^1_{A_1,0}(\Omega_1)\times\boldsymbol  H^1_{A_2,0}(\Omega_2)} dt \\ \leq \tilde{C}_1:= \frac{e^{2M_2 T}-1}{2M_2}\|u^1_n(.,0),u^2_n(.,0)\|^2_{\boldsymbol L^2(\Omega_1)\times \boldsymbol L^2(\Omega_2)}
			\end{split}
		\end{equation}
		The inequality \eqref{estUn2} implies that $( u_n^1, u_n^2)_n$ is a bounded sequence in\\ $\boldsymbol L^2(0,T,\boldsymbol H^1_{A_1,0}(\Omega_1))\times \boldsymbol L^2(0,T,\boldsymbol H^1_{A_2,0}(\Omega_2))$ and bounded in \\$L^{\infty}\big(0,T,\boldsymbol L^2(\Omega_1))\times L^{\infty}\big(0,T,\boldsymbol  L^2(\Omega_2)\big)$ using the operator $\mathcal{I}$.

		\textbf{Step 3: Estimation of $ (\partial_{t}u_n^1,\partial_{t}u_n^2)$}: 
		\\ Fix any \( v_i \in \boldsymbol H^1_{A_i,0}(\Omega_i) \), with \( \| v_i \|_{\boldsymbol H^1_{A_i,0}(\Omega_i)} \leq 1 \)  and $i=1,2$,  such that \( v_i = v^1_i + v^2_i \), where 
		\( v^1_i \in \operatorname{span} \{ v_{i,k} \}_{k=1}^{n} \) and \( (v^2_i, v_{i,k}) = 0 \) (\( k = 1, \dots, n \)). Since the functions 
		\(\{ v_{i,k} \}_{k=0}^{\infty} \) are orthogonal in \(  \boldsymbol H^1_{A_i,0}(\Omega_i) \), we have \( \| v^1_i \|_{\boldsymbol H^1_{A_i,0}(\Omega_i)} \leq \| v_i \|_{\boldsymbol H^1_{A_i,0}(\Omega_i)} \leq 1 \). \\
		Utilizing \eqref{FV}, we deduce :
		\begin{equation}\label{FVt0}
			\begin{split}
				\int_{\Omega_1}  \partial_{t}  u_n^1 v^1_1 = -\int_{\Omega_1}A_1(\lambda_1) \boldsymbol\nabla    u_n^1  \boldsymbol\nabla v^1_1 -&\int_{\Omega_1}A_2(\lambda_2) \boldsymbol\nabla    u_n^2  \boldsymbol\nabla v^1_1  +\int_{\Omega_1}    g(  u_n^1 ,u_n^2) v^1_1  +\\& \int_{\Gamma}    \alpha(I^1_n) u_n^1  v_1^1+\int_{\Omega_1} \lambda_2   u_n^2 v^1_1 -\int_{\Omega_1} \lambda_1   u_n^1  v^1_1 ,   \\
				\int_{\Omega_2}  \partial_{t}  u_n^2  v^1_2  =-\int_{\Omega_2}A_2(\lambda_2) \boldsymbol\nabla    u_n^2  \boldsymbol\nabla v^1_2-&\int_{\Omega_2}A_1(\lambda_1) \boldsymbol\nabla   v^1_2 \boldsymbol\nabla v^1 +\int_{\Omega_2}    g(  u_n^2 ,u_n^1) v^1_2 +\\& \int_{\Gamma}    \alpha(I^2_n) u_n^2  v_2^1+\int_{\Omega_2} \lambda_1  u_n^1 v^1_2-\int_{\Omega_2} \lambda_2   u_n^2 v^1_2, 
			\end{split}
		\end{equation} 
		Then the definition \eqref{defUn} of $u_n$  implies
		\begin{equation}\label{FVt00}
			\begin{split}
				\langle  \partial_{t}u_n^1, v_1  \rangle=\langle  \partial_{t}u_n^1, v^1_1  \rangle = -&\int_{\Omega_1}A_1(\lambda_1) \boldsymbol\nabla    u_n^1  \boldsymbol\nabla v^1_1 -\int_{\Omega_1}A_2(\lambda_2) \boldsymbol\nabla    u_n^2  \boldsymbol\nabla v^1_1  +\int_{\Omega_1}    g(  u_n^1 ,u_n^2) v^1_1  +\\& \int_{\Gamma}    \alpha(I^1_n) u_n^1  v_1^1+\int_{\Omega_1} \lambda_2   u_n^2 v^1_1 -\int_{\Omega_1} \lambda_1   u_n^1  v^1_1 ,   \\
				\langle  \partial_{t} u_n^2, v_2  \rangle=\langle \partial_{t} u_n^2, v^1_2  \rangle =-&\int_{\Omega_2}A_2(\lambda_2) \boldsymbol\nabla    u_n^2  \boldsymbol\nabla v^1_2-\int_{\Omega_2}A_1(\lambda_1) \boldsymbol\nabla   v^1_2 \boldsymbol\nabla v^1_2 +\int_{\Omega_2}    g(  u_n^2 ,u_n^1) v^1_2 +\\& \int_{\Gamma}    \alpha(I^2_n) u_n^2  v_2^1+\int_{\Omega_2} \lambda_1  u_n^1 v^1_2-\int_{\Omega_2} \lambda_2   u_n^2 v^1_2, 
			\end{split}
		\end{equation} 
		Consequently
		\begin{equation*}
			\begin{split}
				\vert  \langle \partial_{t} u_n^1, v_1  \rangle\vert \leq C^1(u_n^{1,2}) \| v^1_1 \|_{H^1_{A_1,0}(\Omega_1)}    \\
				\vert  \langle \partial_{t} u_n^2, v_2  \rangle\vert \leq C^2(u_n^{1,2}) \| v^1_2 \|_{H^1_{A_2,0}(\Omega_2)}
			\end{split}
		\end{equation*}
		where $$C^1(u_n^{1,2})=\bigg( \Vert  u_n^1 \Vert_{\boldsymbol H^1_{A_1,0}(\Omega_1)}+ \Vert  u_n^2 \Vert_{\boldsymbol H^1_{A_1,0}(\Omega_1)}  + L_g \Vert  u_n^1 \Vert_{\boldsymbol L^2(\Omega_1)}+\Vert  u_n^2 \Vert_{\boldsymbol L^2(\Omega_1)} +\Vert  u_n^1 \Vert_{\boldsymbol L^2(\Omega_1)} \bigg)$$ and 
		$$C^2(u_n^{1,2})=\ \bigg(\Vert  u_n^2 \Vert_{\boldsymbol H^1_{A_2,0}(\Omega_2)}+ \Vert  u_n^1 \Vert_{\boldsymbol H^1_{A_2,0}(\Omega_2)} + L_g \Vert  u_n^2 \Vert_{\boldsymbol L^2(\Omega_2)}+\Vert  u_n^1 \Vert_{\boldsymbol L^2(\Omega_2)} +\Vert  u_n^2 \Vert_{\boldsymbol L^2(\Omega_2)}  \bigg),$$
		using the operator $\mathcal{I} $ and 
		since \( \| v^1_i \|_{H^1_{A_i,0}(\Omega_i)} \leq 1 \). Thus
		\[
		\| ( \partial_{t} u_n^1,\partial_{t} u_n^2) \|_{\boldsymbol H^{-1}_{A_1,0}(\Omega_1)\times \boldsymbol H^{-1}_{A_2,0}(\Omega_2)} \leq \hat C_2  \| (u_n^1,u_n^2) \|_{\boldsymbol H^1_{A_1,0}(\Omega_1)\times \boldsymbol H^1_{A_2,0}(\Omega_2)} ,
		\]
		and therefore
		\begin{equation}\label{ut_esti}
			\begin{split}
				\Vert ( \partial_{t} u_n^1,\partial_{t} u_n^2 )\Vert_{\boldsymbol L^2(0,T,\boldsymbol H^{-1}_{A_1,0}(\Omega_1))\times \boldsymbol L^2(0,T,\boldsymbol H^{-1}_{A_2,0}(\Omega_2))}=\int_0^T( \partial_{t} u_n^1,\partial_{t} u_n^2) \|^2_{\boldsymbol H^{-1}_{A_1,0}(\Omega_1)\times \boldsymbol H^{-1}_{A_2,0}(\Omega_2)}  dt\\ \leq \hat C_2 \int_0^T   \| (u_n^1,u_n^2) \|^2_{\boldsymbol H^1_{A_1,0}(\Omega_1)\times \boldsymbol H^1_{A_2,0}(\Omega_2)} dt \leq \hat C_2 \hat C_1
			\end{split}
		\end{equation}
		the inequality \eqref{ut_esti} implies that $(\partial_{t} u_n^1,\partial_{t} u_n^2)_n$ is bounded in $\boldsymbol L^2(0,T,\boldsymbol H^{-1}_{A_1,0}(\Omega_1))\times \boldsymbol L^2(0,T,\boldsymbol H^{-1}_{A_2,0}(\Omega_2)).$
		%%%

		%%%
		\textbf{Step 4: Convergence} \\
		There exists a subsequence, still denoted $(  u_{n}^1,  u_{n}^2)$, such that
		\begin{equation}
			u_{n}^i\rightharpoonup   u^i \;\;\text{weakly in}\; \boldsymbol L^2(0,T,\boldsymbol H^{1}_{A_i,0}(\Omega)),
		\end{equation}

		\begin{equation}
			\partial_{t}  u_{n}^i\rightharpoonup \partial_{t}  u^i \;\;\text{weakly in}\;\boldsymbol L^2(0,T,\boldsymbol H^{-1}_{A_i,0}(\Omega)),
		\end{equation}
		
		\textbf{Step 5: Passage to the limit} \\
		For all $i,j=1,2 $ and $i\neq j$ we have 
		\begin{equation*}
			\begin{split}
				\int_{0}^{T}\langle \partial_{t}  u_{n}^i,v^i  \rangle_{\Omega_i}+\int_{0}^{T}\langle A_i(\lambda_i)  \boldsymbol \nabla  u_{n}^i,\nabla v^i \rangle_{\Omega_i}+\int_{0}^{T}\langle A_j(\lambda_j)  \boldsymbol \nabla  u_{n}^j,\nabla v^i \rangle_{\Omega_i}= \int_{0}^{T}\langle g(u_{n}^i,u_{n}^j), v^i \rangle_{\Omega_i} +\\ 
                 \int_{0}^{T}\langle \alpha(I_{n}^i) u_{n}^i, v^i \rangle_{\Gamma} +\int_{0}^{T}  \langle \lambda_j u_n^j, v^i \rangle_{\Omega_i}-\int_0^{T}\langle \lambda_i  u_n^i,  v^i \rangle_{\Omega_i}.
			\end{split}
		\end{equation*}
		Since $g$ is globally Lipschitz continuous and $(  u_{n}^i)_{n}$ is bounded as above,
	By the Aubin-Lions lemma \cite{temam2024navier},
		$$g(  u_{n}^i,u_{n}^j)\rightarrow g(  u^i,u^j) \text{ strongly in }  \boldsymbol L^{2}\big(0,T,\boldsymbol L^{2}(\Omega_i)\big).$$
Similarly, this approach yields the following convergence result :
	$$\alpha (  I_{n}^i)\rightarrow \alpha (  I^i) \text{ strongly in }  \boldsymbol L^{2}\big(0,T,\boldsymbol L^{2}(\Gamma)\big), \quad \text{for } i=1,2.$$
		and passing to the limit when $n\rightarrow+\infty$ in the weak formulation, for  all $i,j=1,2 $ and $i\neq j$,  we obtain
		\begin{equation*}
			\begin{split}
				\int_{0}^{T}\langle \partial_{t}  u^i,v^i  \rangle_{\Omega_i}+\int_{0}^{T}\langle A_i(\lambda_i)  \boldsymbol \nabla  u^i,\nabla v^i\rangle_{\Omega_i}+\int_{0}^{T}\langle A_j(\lambda_j)  \boldsymbol \nabla  u^j,\nabla v^i\rangle_{\Omega_i}= \int_{0}^{T}\langle g(u^i,u^j), v^i\rangle_{\Omega_i} +\\   \langle \int_{0}^{T} \alpha(I^i) u^i, v^i \rangle_{\Gamma}+\int_{0}^{T}  \langle \lambda_j u^j ,v^i\rangle_{\Omega_i}-\int_ {0}^{T}\langle \lambda_i  u^i,  v^i\rangle_{\Omega_i}.
			\end{split}
		\end{equation*}
		for all $ v \in  \boldsymbol L^2(0,T,\boldsymbol H^{1}_{A_1,0}(\Omega))\times\boldsymbol L^2(0,T,\boldsymbol H^{1}_{A_2,0}(\Omega)) $.
		
		\textbf{Step 6: Uniqueness } \\
		Let us suppose that the problem  $\eqref{abrP}$ has two solutions $(u^1,u^2)$ and $(w^1,w^2)$. Subtracting the weak formulations $\eqref{FV}$ of the solutions yields 
		\begin{equation*}
			\begin{split}
				&\int_{\Omega_1} (\partial_{t}  u^1-\partial_{t}  w^1) v_1  +\int_{\Omega_1}A_1(\lambda_1)( \boldsymbol\nabla 
				u^1-\boldsymbol\nabla w^1)\nabla v_1 +\int_{\Omega_1}A_2(\lambda_2)( \boldsymbol\nabla u^2-\boldsymbol\nabla w^2)\nabla v_1+\\
				&\int_{\Omega_2} (\partial_{t}  u^2-\partial_{t}  w^2) v_2  +\int_{\Omega_2}A_1(\lambda_1)( \boldsymbol\nabla 
				u^1-\boldsymbol\nabla w^1)\nabla v_2 +\int_{\Omega_2}A_2(\lambda_2)( \boldsymbol\nabla u^2-\boldsymbol\nabla w^2)\nabla v_2  \\= &\int_{\Omega_1} \lambda_1  ( u^2-w^2)  v_1-\int_{\Omega_1} \lambda_1  ( u^1-w^1)  v_1+\int_{\Omega_2} \lambda_2  ( u^1-w^1)v_2-\int_{\Omega_2} \lambda_2  ( u^2-w^2) v_2,\\
				&+ \int_{\Omega_1}v_1\left( g(u^1,u^2)-g(w^1,w^2)\right) + \int_{\Omega_2}v_2\left( g(u^2,u^1)-g(w^2,w^1)\right) +  \sum_{i=1}^2\int_\Gamma \alpha(I_{u}^i) u^i-\alpha(I_{w}^i) w^i, v^i \rangle_{\Gamma}
			\end{split}
		\end{equation*}
		by taking $v_i = u^i- w^i, i=1,2$ as a test function gives
		\begin{equation*}
			\begin{split}
				&\int_{\Omega_1}\partial_{t}v_1 v_1+\int_{\Omega_1}A_1(\lambda_1) \boldsymbol \nabla v_1  \boldsymbol\nabla v_1 +\int_{\Omega_1}A_2(\lambda_2) \boldsymbol \nabla v_1  \boldsymbol\nabla v_2+
				\int_{\Omega_1}\partial_{t}v_2 v_2+\int_{\Omega_2}A_2(\lambda_2) \boldsymbol \nabla v_2  \boldsymbol\nabla v_2 \\&+\int_{\Omega_2}A_1(\lambda_1) \boldsymbol \nabla v_1  \boldsymbol\nabla v_2
				=\int_{\Omega_1} \lambda_1  v_2  v_1-\int_{\Omega_1} \lambda_1   v_1  v_1+\int_{\Omega_2} \lambda_2   v_1v_2-\int_{\Omega_2} \lambda_2   v_2 v_2,\\
				&+ \int_{\Omega_1}v_1\left( g(u^1,u^2)-g(w^1,w^2)\right) + \int_{\Omega_2}v_2\left( g(u^2,u^1)-g(w^2,w^1)\right) +  \sum_{i=1}^2\int_\Gamma \alpha(I_{u}^i) u^i-\alpha(I_{w}^i) w^i, v^i \rangle_{\Gamma}
			\end{split}
		\end{equation*}
		using the same technique the estimation of the approach solution sequence $u_n$  we obtain  the following estimation  similar to the estimation \eqref{eqb1}
		\begin{equation}\label{eqbphi1}
			\begin{small}
				\frac{1}{2} \frac{d }{d t}\Vert ( v_1 (.,s),v_2(.,s)) \Vert^2_{\boldsymbol L^2(\Omega_1)\times \boldsymbol L^2(\Omega_2)}+\Vert ( v_1,v_2) \Vert^2_{\boldsymbol H^1_{A_1,0}(\Omega_1)\times\boldsymbol  H^1_{A_2,0}(\Omega_2) } \leq  M_{v}  \Vert (v_1 ,v_2)  \Vert^2_{\boldsymbol L^2(\Omega_1)\times \boldsymbol L^2(\Omega_2) }
		\end{small}
		\end{equation}
		Applying the  Gronwall's lemma  gives
		$$v_i = 0 \; \text{ a. e. in } Q_i.$$	
	\end{proof}

	before proving the theorem \eqref{thoe1}, we know that 
	the reaction function $f:\boldsymbol H^1_{A_1}(\Omega)\times \boldsymbol H^1_{A_2}(\Omega)\rightarrow \boldsymbol L^{2}(\Omega)\times \boldsymbol L^{2}(\Omega)$ is locally Lipschitz, we use a standard truncation of the source as follows, for all $u=(u^1,u^2)$:
	\begin{equation*}
		\ds \mathcal T r( u)=\left\{
		\begin{array}{ll}
			f(  u),\;\;\mbox{if}\;\;\; \|  u\|_{\boldsymbol H^1_{A_1}(\Omega)\times \boldsymbol H^1_{A_2}(\Omega)}\leq \delta&  \\[2ex]
			f\bigg(\frac{\delta   u}{\|  u\|_{\boldsymbol H^1_{A_1}(\Omega)\times \boldsymbol H^1_{A_2}(\Omega)}}\bigg),\;\;\mbox{if}\;\;\; \|  u\|_{\boldsymbol H^1_{A_1}(\Omega)\times \boldsymbol H^1_{A_2}(\Omega)}> \delta &
		\end{array}
		\right.
	\end{equation*}
	where $\delta$ is a positive constant. it is easy to see that $\mathcal T r $ is globally Lipschitz with a Lipschitz constant $L_{\mathcal T r}>0$ and  it satisfies $\mathcal T r (0,0)=0$.
	\begin{proof}[Proof of the theorem \eqref{thoe1}] The function $\mathcal T r$ verifies the same hypotheses of the function $g$ define in the proposition \eqref{prop3}, which means that the following truncated problem has a unique solution $u_{\delta}=(u_{\delta}^1,u_{\delta}^2)$  
		\begin{equation}\label{abrP1}
			\left\{
			\begin{array}{ll}
				\partial_{t}u_{\delta}^i-\mathrm{div}(A_i(\lambda_i)\boldsymbol\nabla u^i_{\delta}) =\mathcal T r(u^i_{\delta},u^j_{\delta})+\mathrm{div}(A_j(\lambda_j)\boldsymbol\nabla u^j_{\delta})+\lambda_j   u^j_{\delta} - \lambda_i   u^i_{\delta},& \text{in } [0,T]\times \Omega_i\\\\
				A_i(\lambda_i)\boldsymbol\nabla u^i_{\delta} \cdot   \tau+A_j(\lambda_j)\boldsymbol\nabla u^j_{\delta} \cdot   \tau+\alpha^j(I^i) u^j_{\delta}=\boldsymbol 0,&\mbox{on}\;\;[0,T]\times\Gamma\\\\
				u^i_{\delta}=\boldsymbol 0&\mbox{on}\;\;[0,T]\times\partial\Omega_i\\\\
				u^i_{\delta}(0,x)=u^i_0(x,\lambda_i,\lambda_j)&  \text{in } \Omega_i\\[2ex]  \text{ for  all } i,j=1,2 \text{ and } i\neq j.
			\end{array}
			\right.
		\end{equation}
		According to the Proposition \eqref{prop3},    $(u_{\delta}^1,u_{\delta}^2)$ satisfies the following energy inequality:
		\begin{equation}\label{esti12}
			\begin{split}
				\Vert  (u_{\delta}^1,u_{\delta}^2) \Vert^2_{\boldsymbol L^2(0,T,\boldsymbol H^1_{A_1,0}(\Omega_1))\times L^2(0,T,\boldsymbol  H^1_{A_2,0}(\Omega_2))} \leq \bar{C}_{\delta}\|u^1_{\delta}(.,0),u^2_{\delta}(.,0)\|^2_{\boldsymbol L^2(\Omega_1)\times \boldsymbol L^2(\Omega_2)}
			\end{split}
		\end{equation}
		Choosing $\bar{C}_{\delta}\|u^1_{\delta}(.,0),u^2_{\delta}(.,0)\|^2_{\boldsymbol L^2(\Omega_1)\times \boldsymbol L^2(\Omega_2)}\leq \delta$ leads to prove $\mathcal T r(u)=f(u)$ given in the initial problem.\\
		Let's now prove the positivity of the solution of our problem, by rewriting  $u^i_\delta$  as follow $u^i_\delta=(\xi^i)^{+}-(\xi^i)^{-}$ where $(\xi^i)^{+}, (\xi^i)^{-}\geq 0$, they are defined by
		$$(\xi^i)^{+}:=\max(0,u^i_\delta) \text{ and }(\xi^i)^{-}:=\max(0,-u^i_\delta)$$ 
		By using the  estimation of solution proven in \eqref{esti12}, we get 
		\begin{equation*}
			\Vert  (\xi^1)^{-}(.,t) \Vert^2_{L^2(\Omega_1)} +  \Vert   (\xi^2)^{-}(.,t) \Vert^2_{L^2(\Omega_2)} \leq \left( \Vert  (\xi^1)^{-}(.,0) \Vert^2_{L^2(\Omega_1)} +  \Vert  (\xi^2)^{-} (.,	0) \Vert^2_{L^2(\Omega_2)}\right)\bar{C}_{\delta}.
		\end{equation*}
		Since $u^i_0> 0$ which means that $(\xi^i)^{-}(.,0):=(\xi^i_0)^{-}=0$, last estimation yields to $(\xi^i)^{-}(.,t)=0$ then the positivity of solution is proven.
	\end{proof}

	\section{Finite Volume Method (FVM)} \label{section4}
	
	To approximate the solution of problem \eqref{Pb_main1} by means of the Finite Volume Method (FVM), we discretize both the spatial domain and the temporal interval. The computational domain \(\Omega\) is partitioned into \(N_x\) control volumes (cells) of uniform size \(\Delta x\). The center of the \(i\)-th control volume is denoted by \(x_i\), while its left and right interfaces are represented by \(x_{i-\tfrac{1}{2}}\) and \(x_{i+\tfrac{1}{2}}\), respectively.

	We integrate our problem \eqref{Pb_main1} over a control volume \([x_{i-\tfrac{1}{2}}, x_{i+\tfrac{1}{2}}]\) to obtain  the following finite volume formulation:
	\begin{equation}\label{fvm1}
		\begin{split}
			\int_{x_{i-\tfrac{1}{2}}}^{x_{i+\tfrac{1}{2}}} \partial_t u^i \, dx = \int_{x_{i-\tfrac{1}{2}}}^{x_{i+\tfrac{1}{2}}} \mathrm{div}(A_i(\lambda_i) \boldsymbol\nabla u^i) \, dx + \int_{x_{i-\tfrac{1}{2}}}^{x_{i+\tfrac{1}{2}}} f(u^i, u^j) \, dx \\+ \int_{x_{i-\tfrac{1}{2}}}^{x_{i+\tfrac{1}{2}}} \mathrm{div}(A_j(\lambda_j) \boldsymbol\nabla u^j) \, dx + \int_{x_{i-\tfrac{1}{2}}}^{x_{i+\tfrac{1}{2}}} \lambda_j u^j \, dx - \int_{x_{i-\tfrac{1}{2}}}^{x_{i+\tfrac{1}{2}}} \lambda_i u^i \, dx
		\end{split}
	\end{equation}
	The right-hand side can be simplified using the divergence theorem:
	\[
	\int_{x_{i-\tfrac{1}{2}}}^{x_{i+\tfrac{1}{2}}} \mathrm{div}(A_i(\lambda_i) \boldsymbol\nabla u^i) \, dx = \left[ A_i(\lambda_i) \boldsymbol\nabla u^i \cdot \tau \right]_{x_{i-\tfrac{1}{2}}}^{x_{i+\tfrac{1}{2}}}
	\]
	Let us define the discrete solution \(u^i\) with the following approximated formula:
	\begin{equation}
		U^{i,n}_j \approx u^i(x_j, t_n), \quad \text{for } j = 1, \dots, N_x.
	\end{equation}
	
	In the Finite Volume Method, the fluxes \(\Psi_{i+\tfrac{1}{2}}\) represent the transfer of the conserved quantity across the interfaces of the control volumes. These fluxes play a central role in ensuring local conservation and are typically approximated using suitable numerical schemes. In the case of a degenerate diffusion coefficient \(A_i(\lambda_i)\), the flux at the boundaries of each control volume can be approximated as:

	\[
	\Psi_{i+\tfrac{1}{2}} = A_i(\lambda_i) \frac{u_{i+1} - u_i}{\Delta x}
	\]
	This approximation is based on the first-order finite difference approximation of the gradient of \(u\) across the cell interface.\\ The fluxes \(\Psi_{i+\tfrac{1}{2}}\) are  used to rewrite the problem \eqref{fvm1} as the discrete system of equations in matrix form:
	\begin{equation}
		\frac{d U^{i,n+1}}{dt} + K_i U^{i,n+1} = F_i(U^{i,n+1}, U^{j,n+1}),
		\label{eq:matrix_fvm}
	\end{equation}
	where:
	\begin{itemize}
		\item \(K\) is the matrix, assembled from numerical flux approximations:
		\begin{equation*}
			K_{i,i-1} = -\frac{A_i(\lambda_i)}{(\Delta x)^2}, \quad K_{i,i} = \frac{A_i(\lambda_i) + A_i(\lambda_i)}{(\Delta x)^2}, \quad K_{i,i+1} = -\frac{A_i(\lambda_i)}{(\Delta x)^2}.
			\label{eq:stiffness_matrix}
		\end{equation*}
		\item \(F_i(U^{i,n+1}, U^{j,n+1})\) contains nonlinear terms and coupling effects.
	\end{itemize}
	To advance the solution in time, we use the implicit Euler method:
	\begin{equation}
		\frac{U^{i,n+1} - U^{i,n}}{\Delta t} + K_i U^{i,n+1} = F_i(U^{i,n+1}, U^{j,n+1}),
		\label{eq:time_step_fvm}
	\end{equation}
	where \(\Delta t\) is the time step size. 
	
	The coupled system is solved iteratively using the following steps:
	\begin{enumerate}
		\item \textbf{Solve for \(U^{1,n+1}\):}
		\begin{equation}
			\frac{U^{1,n+1} - U^{1,n}}{\Delta t} + K_1 U^{1,n+1} = F_1(U^{1,n}, U^{2,n}).
			\label{eq:solve_u1_fvm}
		\end{equation}
		\item \textbf{Use the solution of \(U^{1,n+1}\) to solve for \(U^{2,n+1}\):}
		\begin{equation}
			\frac{U^{2,n+1} - U^{2,n}}{\Delta t} + K_2 U^{2,n+1} = F_2(U^{2,n}, U^{1,n+1}).
			\label{eq:solve_u2_fvm}
		\end{equation}
	\end{enumerate}
	
	These steps ensure that the solution at each time step remains consistent with the governing equations and coupling conditions.
	
	\section{Numerical Results}\label{section5}
	
	This section presents the numerical results of the model in a one-dimensional (1D) spatial domain. The domain is divided into two subdomains: $\Omega_1 = ]0, 1]$ and $\Omega_2 = [1, 2[$. Each subdomain is discretized using $N_x = 302$ spatial points. The simulations are conducted over a time interval of $T = 300$ days, with a time step size of $\Delta t = 0.0125$ days. They are performed using the specified initial conditions and parameter values detailed below.

    The initial population in each subdomain is set as follows:
\begin{itemize}
    \item \textbf{In the subdomain $\Omega_1 =  \, ]0, 1]$:} The susceptible population ($S_1$) is uniformly distributed with an initial normalized value of $0.8$ ($240$ individuals), the infected population ($I_1$) is uniformly distributed with an initial normalized value of $0.2$ ($60$ individuals), and the recovered population ($R_1$) is initially zero.

    \item \textbf{In the subdomain $\Omega_2 = [1, 2[$:} The susceptible population ($S_2$) is uniformly distributed with an initial normalized value of $1.0$ ($300$ individuals), while both the infected population ($I_2$) and the recovered population ($R_2$) are initially zero.
\end{itemize}

	The parameter\(\lambda_j\) denotes the probability that an individual originating from subdomain \(\Omega_j\) resides in subdomain \(\Omega_i\), thereby accounting for immigration before disease detection. To model this probability, we employ, for instance, a spatial Poisson process, as described in \cite{Privault2013}. This choice is motivated by the need to capture both the spatial variability of population's movements and the inherent randomness of individual migration patterns. A spatial Poisson process provides a flexible framework for representing the distribution of individuals across the domain, thereby allowing the incorporation of spatial heterogeneity into the infection dynamics. Through this approach, we can simulate how the probability of an individual being located in a given subdomain varies spatially, thus reflecting more realistic scenarios where population movement is inherently non-uniform.

	For our simulations, we consider a simple yet illustrative example where
	% $\lambda_i$ is set to zero at the boundaries ($x = 0$ and $x = 2$), and a constant value $\lambda$ elsewhere in the domain. 
	we define the probability function $\lambda_i(x)$ as:
	\[
	\lambda_1(x) = 
	\begin{cases}
		0, & \text{if } x = 0, \\
		\lambda, & \text{otherwise}.
	\end{cases}
	\quad
	\text{ and }\quad\lambda_2(x) = 
	\begin{cases}
		%0, & \text{if } x = 0 \text{ for } i = 1, \\
		0, & \text{if } x = 2 \\
		\lambda, & \text{otherwise}.
	\end{cases}
	\]

	This configuration enables us to investigate the impact of uniform immigration probabilities in the interior regions while taking into account potential barriers or restrictions at the domain edges.

	The key parameters used in the simulations are summarized in Table~\ref{tab:parameters}. These parameters are critical in determining the dynamics of the model and are chosen based on realistic epidemiological scenarios.
	
	\begin{table}[h!]
		\centering
		\caption{Model Parameters}
		\begin{tabular}{|l|c|c|c|c|c|c|c|c|c|c|c|}
			\hline
			\textbf{Parameter} & $\beta_1$ & $\beta_2$ & $\beta_{ij}$ & $\gamma_1$ & $\gamma_2$ & $\Lambda_1$ & $\Lambda_2$ & $\mu_S$ & $\mu_I$ & $\mu_R$& $\lambda$ \\ \hline
			\textbf{Value (day$^{-1}$)} & 0.05 & 0.05 & 0.1 & 0.2 & 0.2 & 0.005 & 0.005 & 0.05 & 0.13 & 0.05&0.01 \\ \hline
		\end{tabular}
		\label{tab:parameters}
		
	\end{table}
	
	The diffusion coefficient function \(\sigma(y,t,\lambda)\) is a critical component in modelling the spread of the disease. It describes how the disease spreads through the spatial domain over time. The choice of \(\sigma(y,t,\lambda)\) can significantly influence the model's behavior. Here, we present the chosen form of the diffusion coefficient function
	\begin{equation}
		\sigma(y,t,\lambda) = \lambda (2 - y) y \cdot e^{-a \cdot (t - t_a)}
	\end{equation}
	where $a$ and $t_a$ are given parameters
	\begin{figure}[H]
		\centering
		\includegraphics[width=0.8\textwidth]{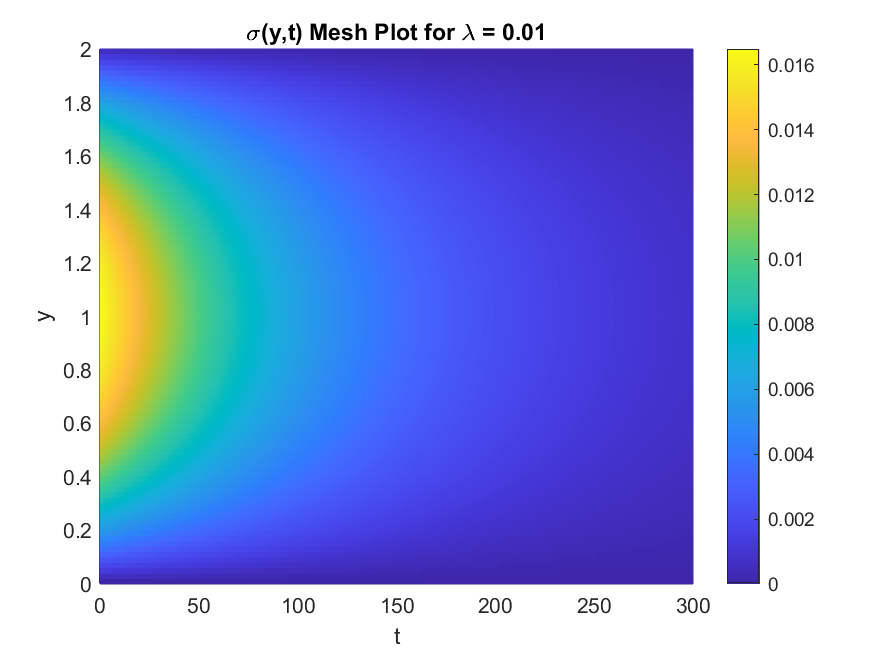}
		\caption{Diffusion coefficient function \(\sigma(y,t)\) for \(\lambda = 0.01, a=0.01\) and $t_a=50$. }
		\label{fig:sigma_lambda_0.01}
	\end{figure}
	This form ensures that the diffusion coefficient decreases exponentially over time, as shown in the figure \ref{fig:sigma_lambda_0.01} starting from a peak at \( t = t_a \) days. The term \(e^{-a \cdot (t - t_a)}\) captures this temporal decay. The spatial variation is governed by the term \((2 - y) y\), which is a quadratic function of \( y \). This quadratic term ensures that the diffusion coefficient is highest at \( y = 1 \) and decreases as \( y \) approaches 0 or 2. This choice is particularly useful for modelling scenarios where the diffusion rate decreases over time and has a peak at a specific point in the spatial domain, symmetrically decreasing as you move away from that point.

	 The figures illustrate the SIR dynamics for regions  $\Omega_1$ and Region $\Omega_2$ over time, showing both the general temporal trends and the evolution at fixed spatial positions within each domain. This dual approach highlights how the combined susceptible, infected, and recovered populations change over time at specific locations. The results are further supported by heatmaps (Figures (\ref{fig:S_heatmap}),(\ref{fig:I_heatmap}) and (\ref{fig:R_heatmap})), which provide a comprehensive spatial and temporal visualization of the solution for each compartment, clearly showing the distribution and evolution of susceptible, infected, and recovered populations across the entire domain.
		Figure (\ref{figN}) shows the total population's evolution over time for both regions, highlighting an initial increase due to births and immigration, followed by stabilization as these factors balance with migration and deaths, reaching equilibrium.

	\begin{figure}[H]
		\centering
		\begin{subfigure}{0.45\textwidth}
			\includegraphics[width=\linewidth]{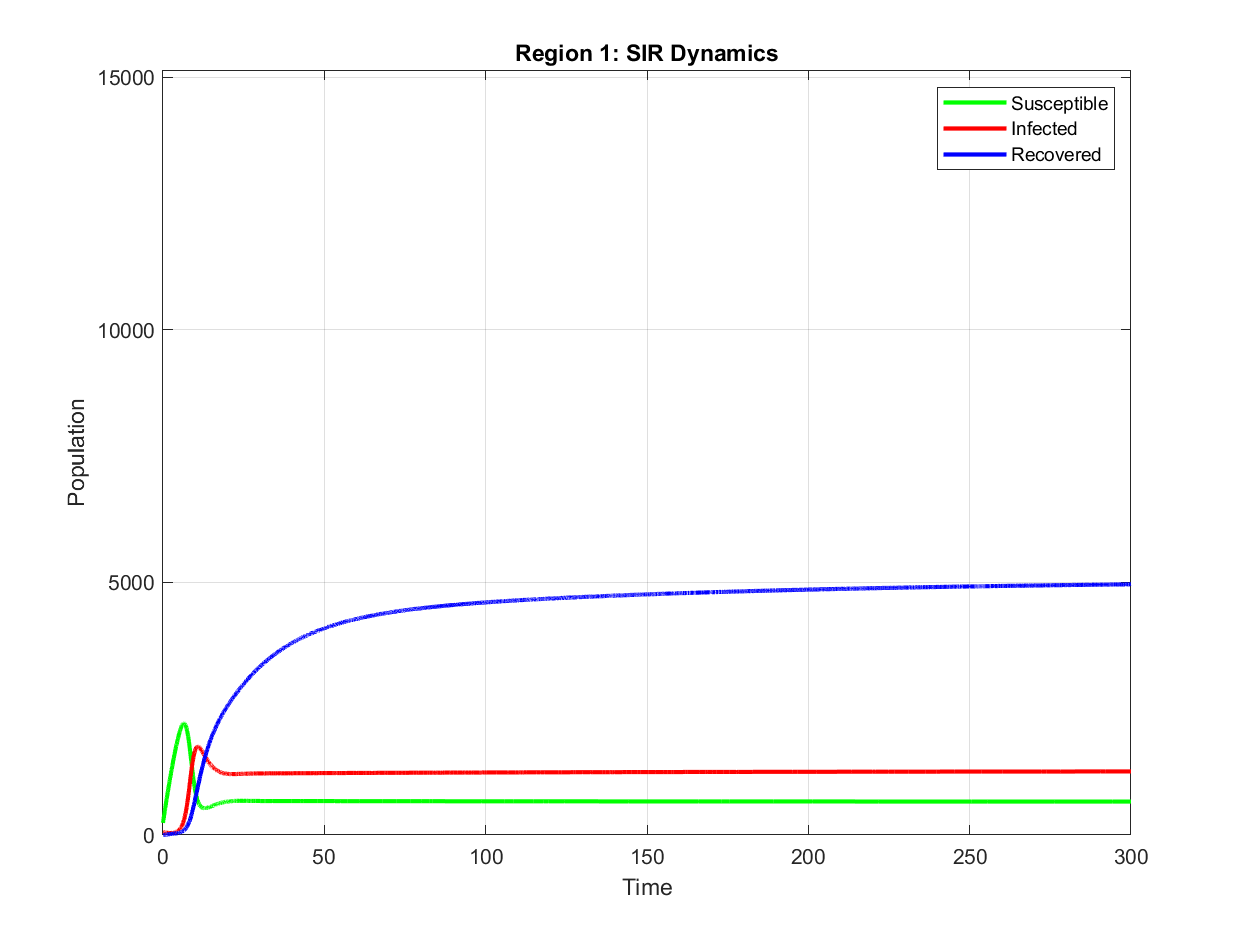}
			\caption{SIR in $\Omega_1$}
		\end{subfigure}
		\hfill
		\begin{subfigure}{0.45\textwidth}
			\includegraphics[width=\linewidth]{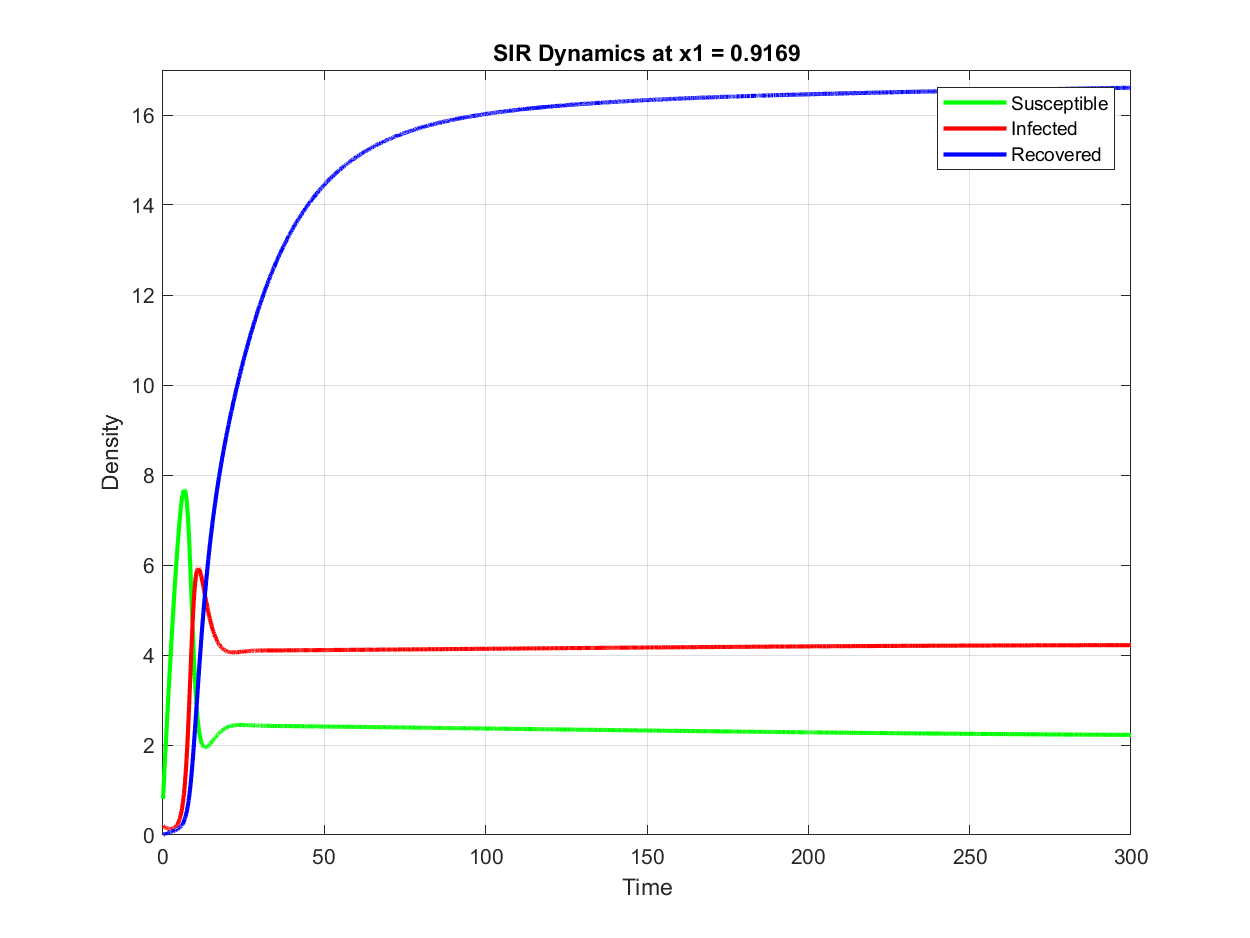}
			\caption{SIR evolution in a fixed point in $\Omega_1$}
		\end{subfigure}
		\hfill
		
		\begin{subfigure}{0.45\textwidth}
			\includegraphics[width=\linewidth]{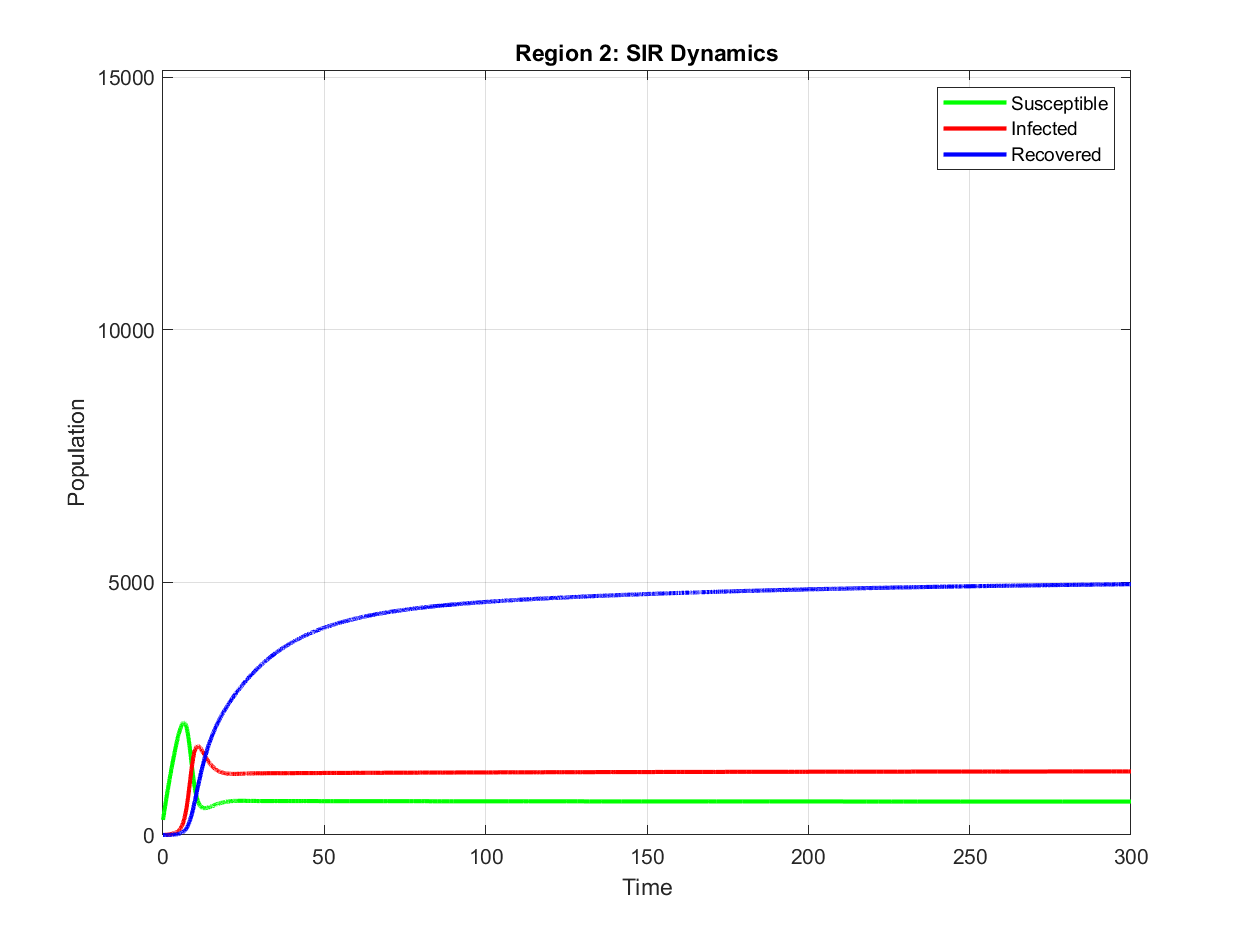}
			\caption{SIR in $\Omega_2$}
		\end{subfigure}
		\hfill
		\begin{subfigure}{0.45\textwidth}
			\includegraphics[width=\linewidth]{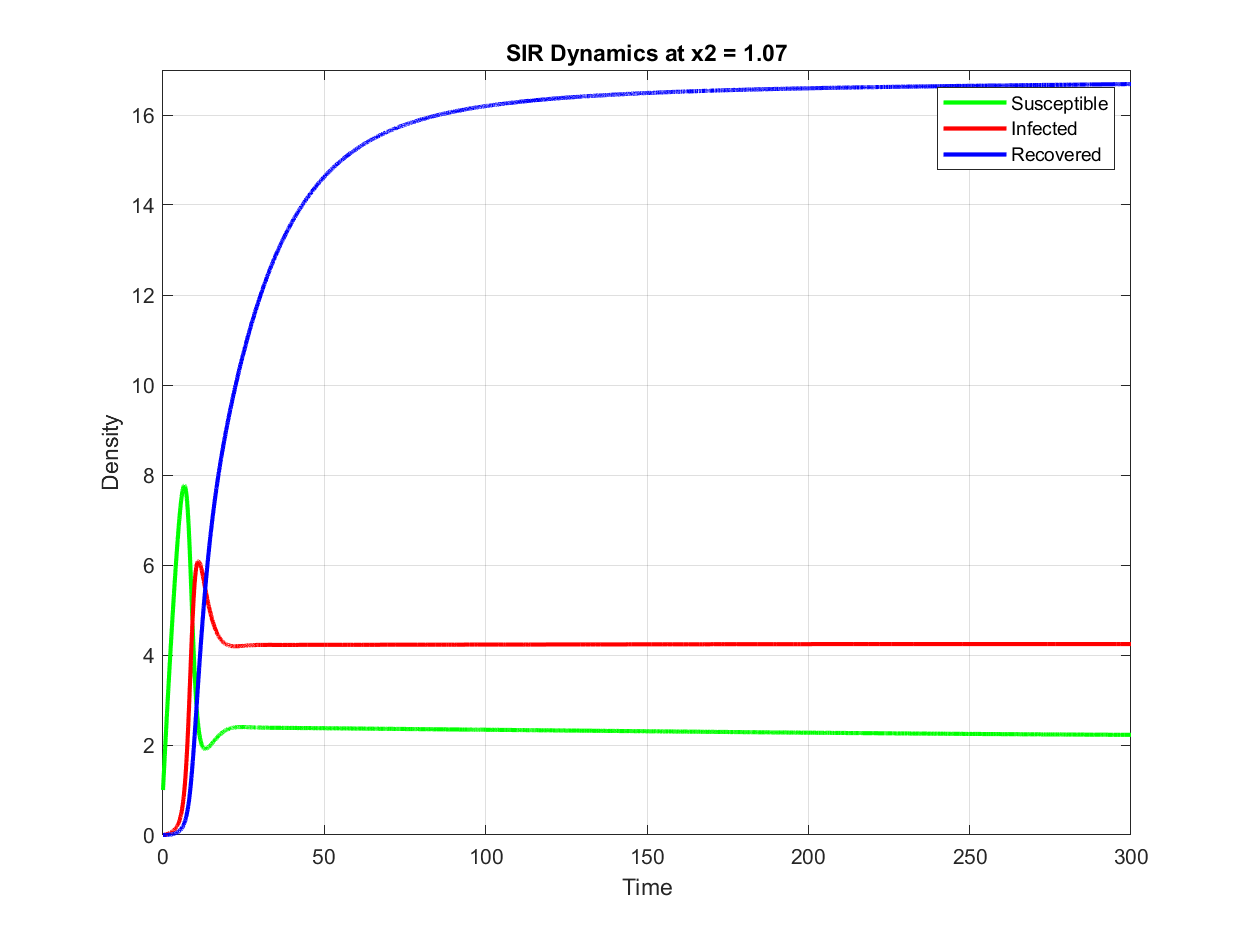}
			\caption{SIR evolution in a fixed point in $\Omega_2$}
		\end{subfigure}
		\caption{SIR model in $\Omega_1$ and $\Omega_2$ }
		\label{fig1}
	\end{figure}
		The table \ref{tab2} demonstrates how varying the parameter $\lambda$  influences key epidemic metrics across different regions. As $\lambda$ increases, the peak number of infected individuals decreases significantly, from 3673 at $\lambda=10^{-5}$
		to 2012 at $\lambda=1$. Similarly, both the total recovered individuals and total population size decrease with higher $\lambda$ values, indicating a more dispersed disease impact. Importantly, lockdown requirements increase dramatically with higher $\lambda$, with lockdown duration extending from approximately 15 days at $\lambda=10^{-5}$ to nearly 260 days at $\lambda=1$. These findings highlight the complex relationship between population movement and disease control in multi-region epidemic models.\\
		Complementing these observations, the contour plots in Figure \ref{fig5} depict the trade-off between epidemic severity and lockdown duration. Moderate values of the mobility parameter strike a desirable balance-substantially reducing the infection peak while maintaining relatively short lockdown periods. In contrast, low mobility results in sharper epidemic peaks but shorter lockdowns, whereas high mobility flattens the curve at the cost of extended restrictions. These results emphasise the critical role of optimizing inter-regional movement to manage both public health and socio-economic impacts effectively.
	
	\begin{figure}[H]
		%\centering
		\begin{subfigure}{0.45\textwidth}
			\includegraphics[width=\linewidth]{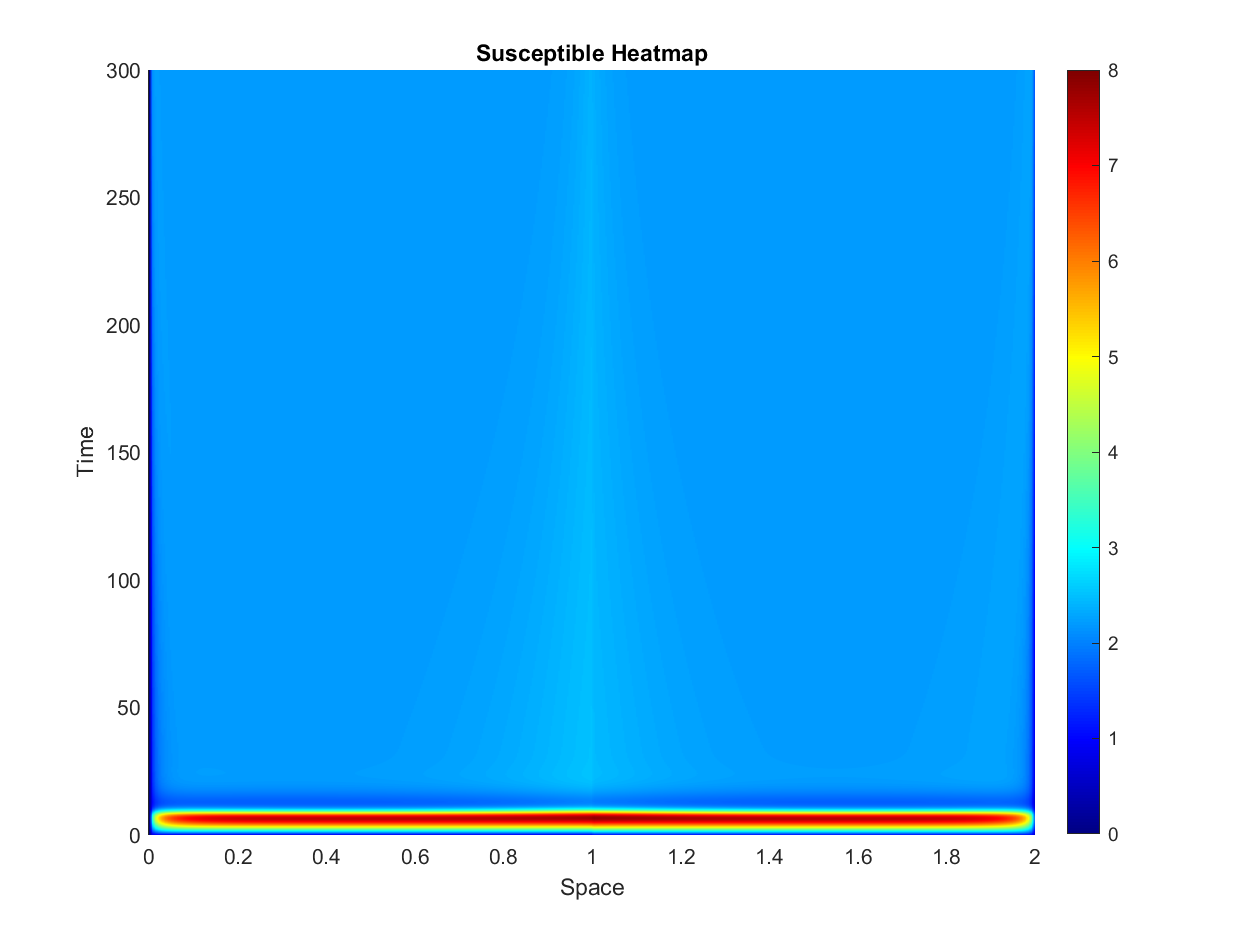}
			\caption{Susceptible Heatmap}
			\label{fig:S_heatmap}
		\end{subfigure}
		\hfill
		\begin{subfigure}{0.45\textwidth}
			\includegraphics[width=\linewidth]{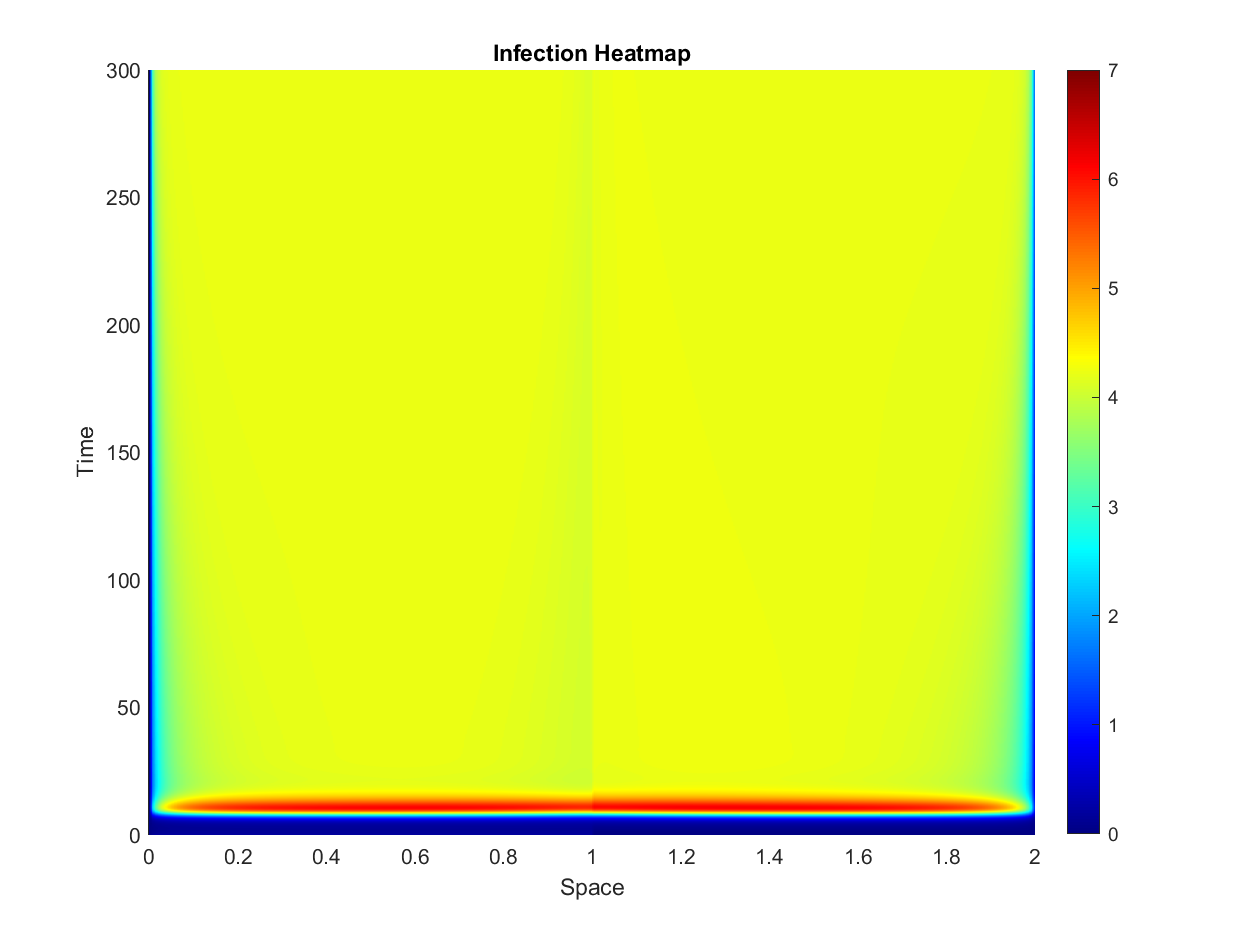}
			\caption{Infection Heatmap}
			\label{fig:I_heatmap}
		\end{subfigure}
		\hfill
		\begin{subfigure}{0.45\textwidth}
			\includegraphics[width=\linewidth]{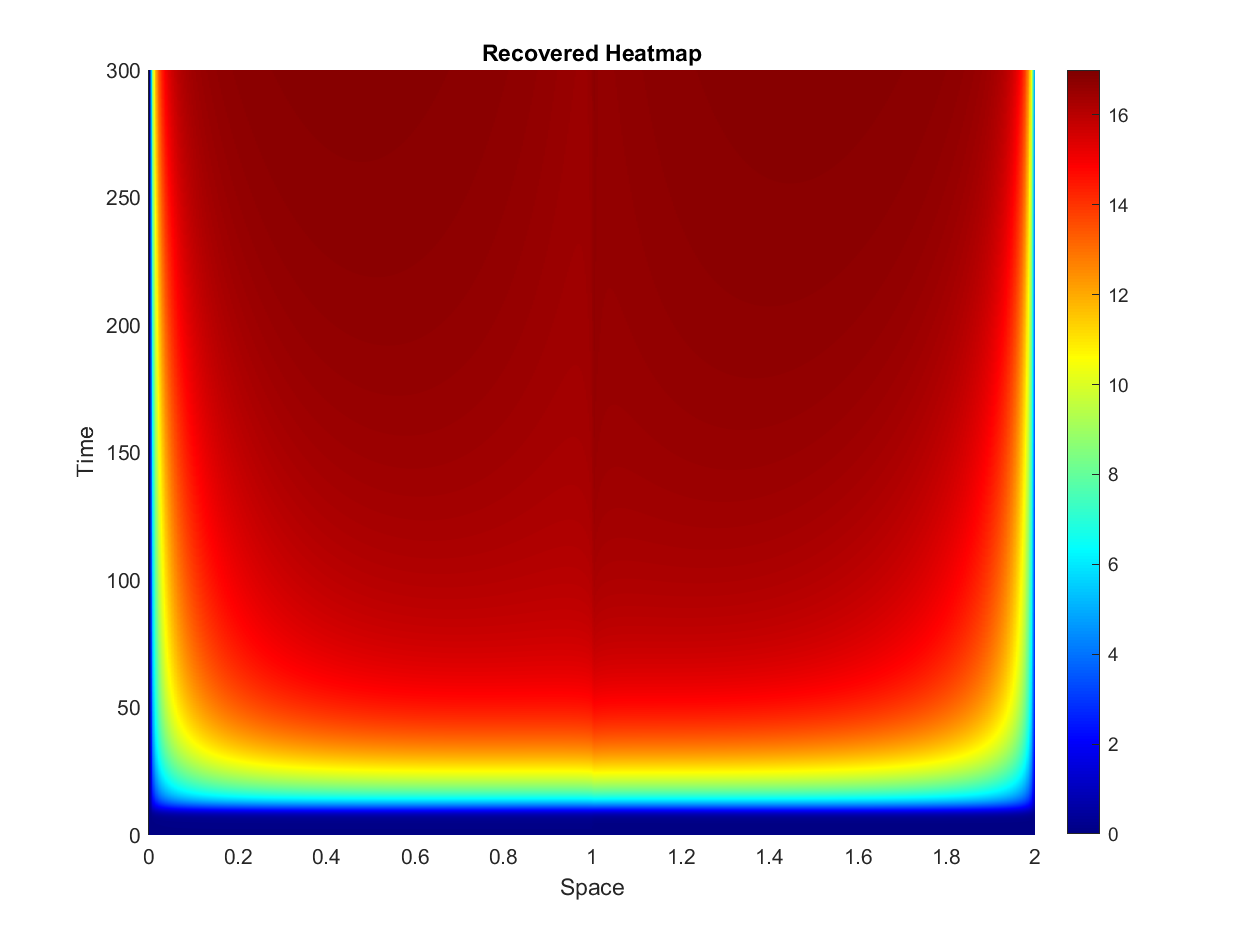}
			\caption{Recovered Heatmap}
			\label{fig:R_heatmap}
		\end{subfigure}
		\hfill
		\begin{subfigure}{0.45\textwidth}
			\includegraphics[width=\linewidth]{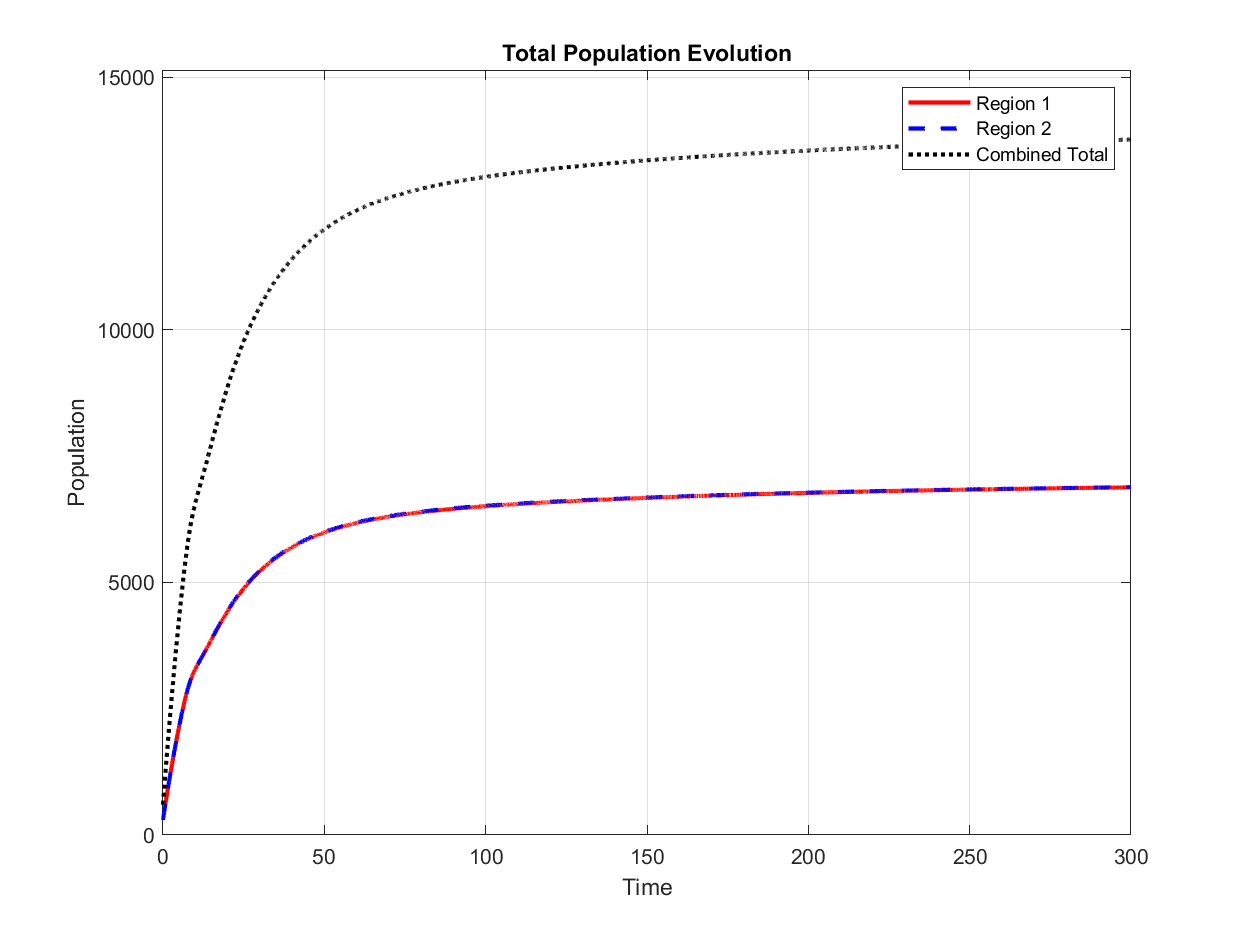}
			\caption{ Total population}
			\label{figN}
		\end{subfigure}
		\label{fig:heatmaps}
		\caption{Heatmaps showing the spatial and temporal evolution of the SIR model, with a figure of the evolution of the total population in each subdomain}
	\end{figure}
	\begin{table}[h]
		\centering
		\caption{Impact of Different $\lambda$ values on proposed  epidemic model  }
		\label{tab2}
		\begin{tabular}{ccccccccc}
			\hline
			\hline
			%\vspace*{10px}
			$\lambda=\lambda_1=\lambda_2$  & $10^{-5}$ & $10^{-4}$  & $10^{-3}$  & $10^{-2}$  & 0.1 & 0.2 & 0.5 & 1 \\
			\hline
			\hline
			\textbf{Total population} & 13979 & 13975 & 13947 & 13788 & 12684 &11709 &9623 & 7548 \\
			\hline
			\textbf{Total recovered} & 10125 & 10121 & 10095 & 9942 & 8880 &7944 &5962& 4037 \\
			\hline
			\textbf{Peak infected}  & 3673 & 3669 & 3645 & 3512 & 2577&2406 & 2251 & 2012 \\
			\hline
			%\hline
			\textbf{Rest infected ($\%$)}  & 18.10 & 18.11 & 18.13 & 18.27 & 19.40&20.55 & 23.39 & 26.65 \\
			\hline
			%\hline
			\textbf{Lockdown (days)}  & 15.61 & 15.62 & 15.74 & 17.34  &69.27&142.11 & 183.8 & 259.71 \\
			\hline
			%\hline
			\textbf{Lockdown ($\%$)}  & 5.20 & 5.21 & 5.24 & 5.78 & 23.09 &47.37& 61.26 & 86.57 \\
			\hline
		\end{tabular}
	\end{table}
	%\section{Varying $\lambda_i$ different} 
	\begin{figure}[H]
		\centering
		\begin{subfigure}{0.47\textwidth}
			\includegraphics[width=\linewidth]{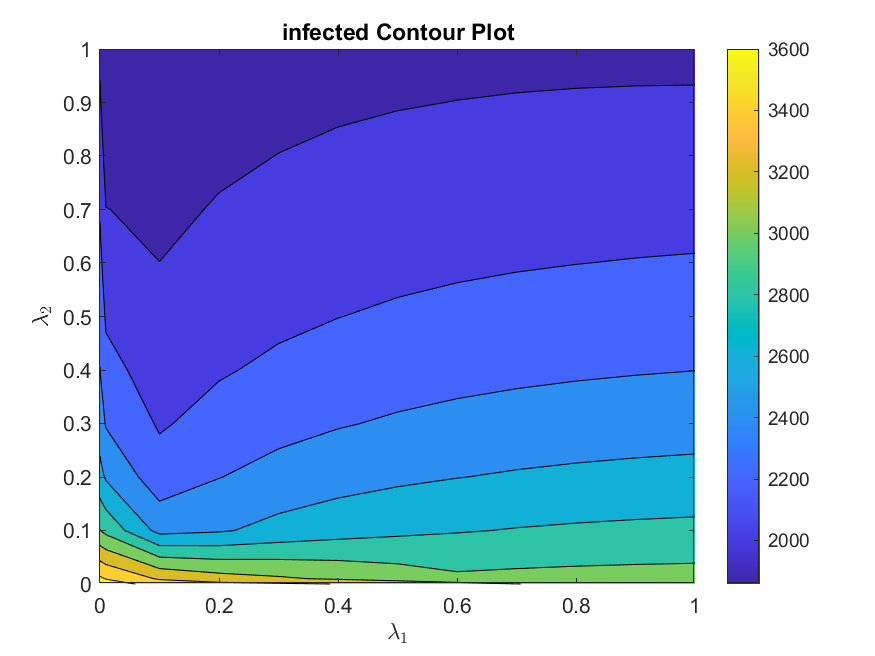}
			\caption{impact of $\lambda_{1,2}$ on the peak infected}
		\end{subfigure}
		\hfill
		\begin{subfigure}{0.47\textwidth}
			\includegraphics[width=\linewidth]{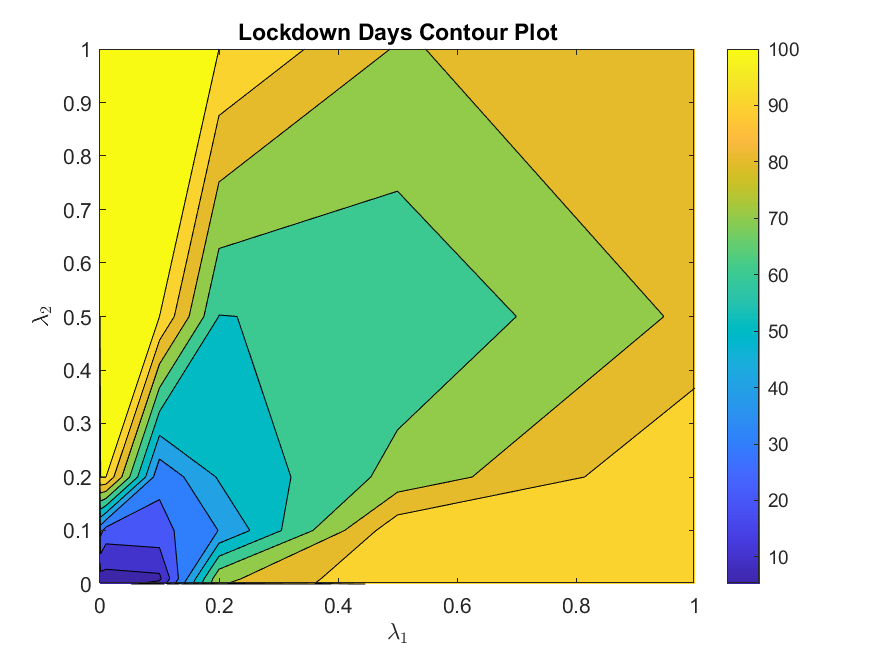}
			\caption{impact of $\lambda_{1,2}$ on Lockdown policy}
			
		\end{subfigure}
		\hfill
		\caption{impact of varying $\lambda_i$, for $i=1,2$ }
		\label{fig5}
	\end{figure}

    \section*{Conclusion}
     
    The numerical simulations presented in this study highlight the rich spatio-temporal dynamics of the proposed degenerate SIR model. By simulating a one-dimensional domain divided into two interacting regions, we demonstrated how population movement, infection thresholds, and policy-driven boundary dynamics jointly influence the epidemic trajectory. The inclusion of degenerate diffusion captures natural mobility constraints, while the transition from Robin to Neumann boundary conditions models the enforcement of lockdowns in a mathematically consistent manner. Simulation results show the temporal evolution of each compartment, confirm theoretical predictions, and reveal the conditions under which regional isolation may mitigate disease spread. These findings underscore the model’s practical value in designing targeted interventions in spatially structured populations.
    \section*{Declaration}

\textbf{Funding} \\
The authors received no specific funding for this research.

\vspace{1em}
\textbf{Conflict of Interest} \\
The authors declare that they have no conflict of interest.

\vspace{1em}
\textbf{Ethical Approval} \\
This research did not involve human participants or animals.

\vspace{1em}
\textbf{Data Availability} \\
No data were generated or analyzed during this study.

\bibliography{Ref_ESZ}
\end{document}